\documentclass{siamart190516}

\usepackage[english]{babel}

\usepackage{stmaryrd}
\usepackage{amssymb}
\usepackage{mathtools}
\newsiamremark{remark}{Remark}
\newsiamremark{assumption}{Assumption}

\newcommand{\R}{\mathbb{R}}
\newcommand{\N}{\mathbb{N}}
\newcommand{\fp}{\mathfrak{p}}
\newcommand{\fm}{\mathfrak{m}}

\newcommand{\om}{\omega}
\newcommand{\lmax}{\lambda_{\text{max}}}

\newcommand{\vp}{\varepsilon}

\DeclareMathOperator{\rank}{rank}
\DeclareMathOperator{\Tr}{Tr}
\DeclareMathOperator{\diag}{diag}
\DeclareMathOperator{\spec}{Spec}
\DeclareMathOperator{\U}{O}

\DeclareMathOperator{\Sy}{Sym}

\DeclareMathOperator{\idty}{Id}

\title{Worst Exponential Decay Rate for Degenerate Gradient flows subject to persistent excitation\thanks{This research was partially supported by the iCODE Institute, research project of the IDEX Paris-Saclay, and by the Hadamard Mathematics LabEx (LMH) through the grant number ANR-11-LABX-0056-LMH in the ``Programme des Investissements d'Avenir''.}}

\author{
Yacine Chitour\thanks{Université Paris-Saclay, CNRS, CentraleSupélec, Laboratoire des signaux et systèmes, 91190, Gif-sur-Yvette, France. Email:  \{\texttt{yacine.chitour, paolo.mason, dario.prandi}\}\texttt{@centralesupelec.fr}}
    \and Paolo Mason\footnotemark[2]
    \and Dario Prandi\footnotemark[2]
}

\begin{document}

\maketitle

%\vspace{-1.5cm}
\begin{abstract}
    In this paper we estimate the worst rate of exponential decay of degenerate gradient flows  $\dot x = -S x$, issued from adaptive control theory \cite{anderson1986}. Under \emph{persistent excitation} assumptions on the positive semi-definite matrix $S$, we provide upper bounds for this rate of decay consistent with previously known lower bounds and analogous stability results for more general classes of persistently excited signals.
  The strategy of proof consists in relating the worst decay rate to optimal control questions and studying in details their solutions. 
  
    As a byproduct of our analysis, we also obtain estimates for the worst $L_2$-gain of the time-varying linear control systems  $\dot x=-cc^\top x+u$, where the signal $c$ is \emph{persistently excited}, thus solving an open problem posed by A. Rantzer in 1999, cf.~\cite[Problem~36]{Rantzer1999}. 
\end{abstract}

\section{Introduction}
The focus of this paper is the convergence rate to the origin associated with descent algorithms of the form
\begin{equation}\label{eq:0}\tag{DGF}
    \dot x(t) = -S(t)x(t), \qquad x\in \mathbb{R}^n,
\end{equation}
where $S$ is a locally integrable positive semi-definite $n\times n$ symmetric matrix.
Whenever $S$ is not positive definite, these dynamics are usually referred to as \emph{degenerate gradient flow systems}.
They appear in the context of adaptive control and identification of parameters (cf.~\cite{Anderson1977,Andersson2002,Brockett2000,Sondhi1976}).
Of particular importance among the dynamics \eqref{eq:0}, is the case where the rank of $S(t)$ is assumed to be at most one, i.e.,  $S=cc^{\top}$ with $c\in \mathbb{R}^n$.

In order to guarantee global exponential stability (GAS) of \eqref{eq:0}, we assume $S$ to satisfy the \emph{persistent excitation} condition. That is, there exists $a,b,T>0$ such that
\begin{equation}\label{PE0}\tag{PE}
    a\idty_n \le \int_{t}^{T+t} S(\tau)\,d\tau \le b\idty_n, \qquad \forall t\ge 0.
\end{equation}
Here, $\idty_n\in \mathbb{R}^{n\times n}$ is the identity matrix, and the inequalities 
are to be understood in the sense of symmetric forms. Clearly, this condition 
is invariant under conjugation by orthogonal matrices and is actually equivalent to 
\emph{uniform} global exponential stability of \eqref{eq:0}, 
cf.~\cite{Anderson1977}. Note also that Condition \eqref{PE0} has been considered 
in stabilization issues for linear control systems with unstable uncontrolled 
dynamics, cf.~\cite{Chitour1,Chitour2}.

Our purpose is to study the \emph{worst exponential decay rate} $R(a,b,T,n)$ of persistently excited signals, as a function of the parameters $a,b,T>0$ and the dimension $n\in \mathbb{N}$. Letting  $\Sy_n^{(PE)}(a,b,T)$ denote the family of signals satisfying \eqref{PE0}, this is defined by
\begin{equation}\label{eq:wedr}
    R(a,b,T,n) = \inf\left\{ R(S) \mid S \in \Sy_n^{(PE)}(a,b,T)\right\},
\end{equation}
where $R(S)$ is the exponential decay rate of \eqref{eq:0}, given in terms of the fundamental matrix $\Phi_S(t,0)$  of \eqref{eq:0} by
\begin{equation}
    R(S) := -\limsup_{t\to +\infty} \frac{\log\|\Phi_S(t,0) \|}{t}.
\end{equation}

The literature on the worst decay rate is extensive (cf., e.g., \cite{Anderson1977, Andersson2002, Brockett2000,Weiss1979}), but mostly restricted to lower bounds. In our context, these results boils down to the existence of an universal constant $C>0$ such that
\begin{equation}\label{est0}
R(a,b,T,n)\geq \frac{Ca}{(1+nb^2)T}.
\end{equation}
Our main result is the following, which shows the optimality of this lower bound, for $n$ fixed.
\begin{theorem}\label{th1}
There exists $C_0>0$ such that, for every $0<a\leq b$, $T>0$ and integer $n\geq 
2$, the worst rate of exponential decay $R(a,b,T,n)$ defined in \eqref{eq:wedr} 
satisfies
\begin{equation}\label{est1}
R(a,b,T,n)\le \frac{C_0 a}{(1+b^2)T}.	
\end{equation}
Moreover, the same result holds true when restricting \eqref{eq:wedr} to matrices 
$S$ verifying \eqref{PE0} with rank at most $1$ (i.e., $S=cc^\top$ and $c\in \R^n$).
\end{theorem}

\begin{remark}\label{rem:barabanov}
	The above shows in particular that $R(a,b,T,n)$ tends to zero as $b$ tends to infinity. This is in accordance with  \cite{Barabanov2005}, where it is proved that in general there is no convergence to the origin for trajectories of \eqref{eq:0} if only the left inequality of  \eqref{PE0} holds true, i.e., $b=+\infty$. More precisely, the authors put forward a ``freezing'' phenomenon by showing that in this case there exist trajectories of \eqref{eq:0} which converge, as $t$ tends to infinity, to points different from the origin.
\end{remark}

\subsection{$L_2$-gain of degenerate flows with linear inputs}

As a consequence of Theorem~\ref{th1} and of the arguments to derive it, we solve the first part of a problem by A.\ Rantzer \cite[Problem 36]{Rantzer1999}, that we now present. 
Consider the control system 
\begin{equation}\label{sys1}
\dot x(t)=-c(t)c(t)^\top x(t)+u(t),
\end{equation}
where $x,c,u$ take values in $\mathbb{R}^n$. 
For $u\in L_2([0,\infty),\mathbb{R}^n)$, let $x_u\in L_2([0,\infty),\mathbb{R}^n)$ be the trajectory of \eqref{sys1} associated with $u$ and starting at the origin.
Whenever $cc^\top$ satisfies~\eqref{PE0}, the trajectories of the uncontrolled dynamics tend to zero exponentially, so that the input/output map $u\mapsto x_u$ is well-defined as a linear operator on $L_2([0,\infty),\mathbb{R}^n)$ 
and its $L_2$-gain $\gamma(c) =\sup_{0\neq u\in L_2([0,\infty),\mathbb{R}^n)}\frac{\|x_u\|_{2}}{\|u\|_{2}}$ is finite. (Here,  $\|\cdot\|_{2}$ stands for the norm in $L_2([0,\infty),\mathbb{R}^n)$.) Rantzer's question consists in estimating
\begin{equation}\label{eq:gamma-gain}
    \gamma(a,b,T,n):=\sup\left\{\gamma(c) \mid c c^\top \text{ satisfy } \eqref{PE0} \right\}.
\end{equation}
In that direction, we obtain the following result.
\begin{theorem}\label{th2}
There exists $c_1,c_2>0$ such that, for every $0<a\leq b$, $T>0$ and integer $n\geq 2$,
\begin{equation}\label{eq:L2G}
c_1 \frac{T(1+ b^2)}{a}\leq \gamma(a,b,n,T)\leq c_2\frac{T(1+n b^2)}{a}.
\end{equation}
\end{theorem}

\begin{remark}
    The same result holds when replacing $cc^\top$ in \eqref{sys1} by a positive semi-definite $n\times n$ symmetric matrix satisfying \eqref{PE0}.
\end{remark}

\subsection{Generalized persistent excitation}

Recently, there has been an increasing interest in considering more general types of persistent excitation conditions, cf.~\cite{Barabanov2017,Praly2017,Efimov2018}. We focus on the following \emph{generalized persistent excitation} condition:
\begin{equation}\label{eq:PEn}\tag{GPE}
a_\ell Id_n\leq \int_{\tau_{\ell}}^{\tau_{\ell+1}}S(t) dt\leq b_\ell\idty_n,
\end{equation}
where $(a_\ell)_{\ell\in\mathbb{N}}$, $(b_\ell)_{\ell\in\mathbb{N}}$ are sequences of positive numbers, and $(\tau_\ell)_{\ell\in\mathbb{N}}$ is a strictly increasing sequence of positive times such that $\tau_\ell\rightarrow +\infty$ as $\ell\to+\infty$. 

An important question consists in determine under which condition \eqref{eq:PEn} guarantees global asymptotic stability (GAS) for \eqref{eq:0}. 
The following sufficient condition is known:
\begin{equation}\label{eq:praly}
    \sum_{\ell=0}^{\infty}\frac{a_\ell}{1+b_\ell^2}=+\infty.
\end{equation}
This has been proved in \cite{Praly2017} (cf.\ also \cite{Barabanov2017}) for the case where $S$ has rank at most one.  The same argument can be extended to the general case, cf., \cite{Brockett2000}.

As a byproduct of our analysis, we show that this condition is indeed necessary.

\begin{theorem}
\label{th3}
All systems \eqref{eq:0} that satisfy condition \eqref{eq:PEn} are GAS if and only if \eqref{eq:praly} holds.

\end{theorem}

We stress that our interest lies in the study of systems satisfying \eqref{eq:PEn} as a class. That is, the above theorem states that if \eqref{eq:praly} is not satisfied, then there exists an input signal satisfying \eqref{eq:PEn} that is not GAS. However, for a fixed signal satisfying \eqref{eq:PEn}, condition \eqref{eq:praly} is not  necessary for GAS, as shown in \cite[Prop.~7]{Barabanov2017}.

\subsection{Strategy of proof}

We now turn to a brief description of the strategy of proof. The main idea is to consider  optimal control problems  whose minimal values provide bounds for the worst-rate of exponential decay.

More precisely, since the dynamics in \eqref{eq:0} are linear in $x\in \R^n$, the system is amenable to be decomposed in spherical coordinates. 
Thus, letting $x = r\omega$, for $r = \|x\|\in \mathbb R_+$ and $\omega=x/\|x\|\in \mathbb{S}^{n-1}$, \eqref{eq:0} reads as 
\begin{eqnarray}
\dot r &=&-r\omega^\top S\omega,\label{sys-r}\\
\dot \omega&=&- S\omega+(\omega^\top S\omega)\omega.\label{sys-om}  
\end{eqnarray}

For $S$ satisfying \eqref{PE0}, consider the control system defined by \eqref{sys1} and let $\Phi_S(\cdot,\cdot)$ be the fundamental matrix associated with 
$S$, i.e., for every $0\leq s\leq t$, $\Phi_S(t,s)$ is the value at time $t$ of the solution of $\dot M=-SM$ with initial condition $M(s)=\idty_n$. 
Observe that, for every $x\in\mathbb R^n\setminus\{0\}$, if we let $\Phi_S(t,0)x = r(t)\omega(t)$,
then it holds
\begin{equation}\label{eq:exp}
\ln\left(\frac{\|\Phi_S(T+t,0)x\|}{\|\Phi_S(t,0)x\|}\right)=\ln \left(\frac{r(t+T)}{r(t)}\right)=-\int_t^{t+T}\omega^\top S\omega \,ds,\quad \forall t\geq 0.
\end{equation} 
Since the last term in the above equation does not depend on $r$, this suggests to consider the optimal control problem 
\begin{equation}\label{ocp0}\tag{OCP}
\inf J(S, \omega_0),\quad J(S, \omega_0):=\int_0^T\omega^\top S\omega\,dt,
\end{equation}
where the infimum is considered among all signals satisfying %~\eqref{PE0} 
\begin{equation}\label{PE00}\tag{INT} 
a \idty_n\leq \int_0^{T}S(\tau) d\tau \leq b\idty_n,
\end{equation}
and initial conditions $\omega_0\in \mathbb{S}^{n-1}$, and $\omega:[0,T]\to \mathbb{S}^{n-1}$ is the trajectory of \eqref{sys-om} with initial condition $\omega(0)=\omega_0$.
In particular, $S$ satisfies \eqref{PE00} if and only if it is the restriction to $[0,T]$ of a signal satisfying \eqref{PE0}.

We show in Proposition~\ref{prop:min} below that \eqref{ocp0} admits minimizers and that the corresponding minimal value is independent of $T$.
Denoting by $\mu(a,b,n)$ this value, in Section~\ref{sec:reduction} we reduce the proof of the main results to the following.

\begin{proposition}\label{prop:main}
There exists a universal constant $C_0>0$ such that, for every $0<a\leq b$ $T>0$ and integer $n\geq 2$,
 \begin{equation}\label{eq:mu2}
   \mu(a,b,n)\le \frac{C_0a}{1+b^2}.
 \end{equation}
 Moreover, there exists a $2T$-periodic rank-one control $S_*=c_*c_*^\top$  such that
 \[2a \idty_n\leq \int_t^{t+2T}S_*(\tau)d\tau \leq 2b\idty_n,\qquad \forall t\geq 0,\]
and an initial condition $\omega_0\in \mathbb S^{n-1}$ such that % such that $c_*|_{[0,T]}$ is an optimal control for \eqref{ocp0} and that there exists an initial condition $\omega_*\in \mathbb{S}^{n-1}$ such that 
 \begin{equation}
%   t\mapsto \frac{\Phi_{c_*}(t,0)\omega_0}{\|\Phi_{c_*}(t,0)\omega_0\|}
	\omega_*(t)= \frac{\Phi_{S_*}(t,0)\omega_0}{\|\Phi_{S_*}(t,0)\omega_0\|}
 \end{equation}
 is a $2T$-periodic trajectory and both $t\mapsto S_*|_{[0,T]}(t)$ and $t\mapsto S_*|_{[T,2T]}(t - T)$, together with the respective initial conditions $\omega_0$ and $\omega_*(T)$, are minimizers for~\eqref{ocp0}. 
\end{proposition}

\begin{remark}
    Observe that the convex hull of rank-one controls satisfying \eqref{PE00} coincides with the considered set of controls.
    Hence, the infimum of $J(S,\omega_0)$ restricted to  controls of the form $S=cc^\top$ is still equal to $\mu(a,b,n)$. The above proposition provides the stronger conclusion that $\mu(a,b,n)$ is actually attained by a rank-one minimizer.
\end{remark}

The rest of the paper is devoted to prove the above proposition.
We first observe that, due to the monotonicity with respect to\ the dimension $n$ of the minimal value  $\mu(a,b,n)$, for the first part of the statement it is enough to bound $\mu(a,b,2)$. 
We then apply Pontryagin  Maximum Principle, and we explicitly integrate the resulting Hamiltonian system in the two dimensional case, thus obtaining the result.
Finally, the proof of the second part of the statement by a detailed analysis in arbitrary dimension $n\geq 2$ of the extremal trajectories associated with \eqref{ocp0}.

\subsection{Notations}

We use $\lfloor x\rfloor$ to denote the integer part of the real number $x$ and $\llbracket a,b \rrbracket$ to denote the set of integers in $[a,b]$.
We let $\Sy_n$ be the set of $n\times n$ symmetric real matrices, and by $\Sy^+_n$ the subset of non negative ones. Moreover, for $a\leq b$, we use $\Sy_n(a,b)$ to denote the set of matrices $Q\in \Sy_n$ such that $a\idty_n\le Q\le b\idty_n$ in the sense of quadratic forms. For every positive integer $k$, we denote by $\mathbb{S}^{k}$ the unit sphere of $\mathbb{R}^{k+1}$.
Finally, we let $\Sy_n^+(a,b,T)$ be the set of functions that satisfy \eqref{PE00}. 
Similarly, we let $\Sy_{n}^{(PE)}(a,b,T)$ be the set of functions that satisfy \eqref{PE0}. 

\section{Preliminary results for the optimal control problem \eqref{ocp0} }

We start  by showing a simple upper bound for the minimal value $\mu(a,b,T,n)$ of \eqref{ocp0}.

\begin{proposition}\label{prop:upper-a}
    It holds
	\begin{equation}
		\mu(a,b,T,n)\le a.
	\end{equation}
\end{proposition}

\begin{proof}
It suffices to consider the matrix $S$ defined by 
\begin{equation}
S(t) = \frac{an}T e_je_j^\top \quad \text{if } t\in \left[ \frac{(j-1)T}{n}, \frac{jT}{n} \right), \qquad j=1,\ldots,n,
\end{equation}	
where $\{e_1,\ldots, e_n\}\subset\R^n$ denotes the canonical basis of $\R^n$. Indeed, 
\begin{equation}
\int_{0}^T S(t)\,dt = a\idty_n
\end{equation}
so that $S$ satisfies~\eqref{PE00}, and,	
for $\omega_0=e_1$, we have that $\omega \equiv \omega_0$ and $J(S,\omega_0)=a$.
\end{proof}

Since we want to apply techniques of optimal control to study the minimal value $\mu(a,b,T,n)$ of \eqref{ocp0}, we
now establish existence of minimizers for such problem.

\begin{proposition}\label{prop:min}
The optimal control problem \eqref{ocp0} admits minimizers with constant trace. Moreover, the minimal value $\mu(a,b,T,n)$ is independent of $T>0$.
\end{proposition}

\begin{proof}
In order to prove the first part of the proposition, we first notice that the infimum in~\eqref{ocp0} remains unchanged if we assume that $S(t)>0$ for all $t\in[0,T]$. Indeed any $S$ satisfying~\eqref{PE00} may be approximated arbitrarily well by a positive definite signal $S_{\vp}=\frac{a}{a+\varepsilon}(S+\varepsilon \idty_n/T)$ with $\varepsilon>0$ still satisfying~\eqref{PE00}, and moreover the functional $J$ depends continuously on the control $S$ (e.g., in the $L_1$ topology) and on the initial condition $\omega_0$ (this may be easily deduced from the continuous dependence on $S$ and $x(0)$ of the original equation~\eqref{eq:0}).

We now show that for any $S>0$ satisfying~\eqref{PE00} there exists a control $\tilde S$ of constant trace satisfying~\eqref{PE00} and such that $J(\tilde S,\omega_0) = J(S,\omega_0)$. Setting ${\cal{T}}= \int_0^T\Tr(S(t))dt,$ we consider the change of time $\tau(t)=\frac{T}{\cal T}\smallint_0^t\Tr(S(s))ds$, which is well defined from $[0,T]$ to itself since $\Tr(S(t))>0$ for any $t\in[0,T]$. If $x$ is the solution of~\eqref{eq:0} with control $S$ it is then easy to see that $\tilde x = x\circ \tau^{-1}$ solves~\eqref{eq:0} with control 
\begin{equation}
\tilde S(\cdot) =  \frac{\mathcal{T}S(\tau^{-1}(\cdot))}{T\Tr(S(\tau^{-1}(\cdot)))},
\end{equation}
so that $J(\tilde S,\omega_0) = J(S,\omega_0)$, and moreover $\int_0^T \tilde S(s)ds = \int_0^T S(s)ds$.

Note now that the set of matrix-valued functions of constant trace in $\Sy_n^+(a,b,T)$  weakly-$\ast$ compact in $L_\infty$ (see Lemma~\ref{lem:weakly-star} in Appendix~\ref{a:weakly-star}). The existence of minimizers with constant trace is then a consequence of the continuous dependence of the functional $J$ on $S$ and $\omega_0$, which in turn may be deduced from the continuous dependence on $S$ (in the weak-$\ast$ topology of $L_{\infty}$) and $x(0)$ of the solutions of~\eqref{eq:0}.

Finally, the independence of $\mu(a,b,T,n)$ from $T$ may be deduced from the fact that, given a solution $x$ of~\eqref{eq:0} corresponding to $S\in \Sy_n^+(a,b,T)$, any time reparametrization $\tilde x$ of $x$  defined on $[0,\tilde T]$ is the solution of~\eqref{eq:0} for some $\tilde S\in \Sy_n^+(a,b,\tilde T)$.
\end{proof}
\begin{remark}
Due to Proposition~\ref{prop:min}, we henceforth let $\mu(a,b,n)=\mu(a,b,T,n)$.  
\end{remark}

The following observation will be crucial in the sequel.

\begin{proposition}\label{prop:mu-n}
  The map $n\mapsto \mu(a,b,n)$ is non-increasing.
\end{proposition}

\begin{proof}
  Consider an admissible trajectory $\omega$ of \eqref{eq:0} in dimension $n$, associated with some $S\in \Sy_n^{+}(a,b,T)$ %$S\in \Sy_n^+(a,b,T)$ 
  and $\omega_0\in \mathbb S^{n-1}$. Then, the trajectory $\tilde \omega = (\omega,0)$ is a trajectory of \eqref{eq:0} in dimension $m\geq n$ associated with $\tilde S=\operatorname{diag} (S, a\idty_{m-n})\in \Sy_m^{+}(a,b,T)$ and initial condition $\tilde \omega_0=(\omega_0,0)$. Due to the form of $\tilde \omega$, we trivially have that $J(S,\omega_0)=J(\tilde S,\tilde \omega_0)$, and thus $\mu(a,b,n)\ge \mu(a,b,m)$.
\end{proof}

\section{Reduction of the main results to Proposition~\ref{prop:main}}
\label{sec:reduction}

In this section, we show that Theorems~\ref{th1}, \ref{th2}, and \ref{th3}, all follow from Proposition~\ref{prop:main}. 
To this aim, we start by determining the homogeneity with respect to\ $T$ of the quantities at hand.

\begin{proposition}\label{prop:hom}
  For every $T>0$ it holds 
  \begin{equation}
    R(a,b,T,n) = \frac{R(a,b,1,n)}{T} 
    \quad\text{and}\quad
    \gamma(a,b,T,n) = T\, \gamma(a,b,1,n)
  \end{equation}
\end{proposition}

\begin{proof}
  If $S \in\Sy_n^{(PE)}(a,b,T)$, then letting $\tilde S(s):=T S(T s)$, we have that $\Phi_{\tilde S}(s,0)=\Phi_S(Ts,0)$ for all $s>0$ and $\tilde S\in\Sy_n^{(PE)}(a,b,1)$. This immediately implies the first part of the statement. 
  On the other hand, if $x(\cdot)$ is the trajectory of \eqref{sys1} associated with $S$ and $u\in L_2((0,+\infty),\R^n)$, then   $x(T\cdot)$ is associated with $\tilde S$  and $\tilde u(s):= Tu(Ts)$. This yields at once that $\gamma(c) = T\gamma(\tilde c)$, completing the proof.
\end{proof}

We are now ready to establish the link between the minimal value $\mu(a,b,n)$ of \eqref{ocp0} and the worst rate of exponential decay for \eqref{eq:0}. We observe that this yields at once the fact that Theorem~\ref{th1} is a consequence of Proposition~\ref{prop:main}.
\begin{proposition}\label{prop:reduc}
It holds that, 
\begin{equation}\label{est-muR}
\frac{\mu(a,b,n)}T\leq R(a,b,T,n)\leq 2\frac{\mu(a/2,b/2,n)}T
\end{equation}
Moreover, the same result holds true when replacing $R(a,b,T,n)$ by the quantity obtained by restricting \eqref{eq:wedr} to rank-one matrices (i.e., $S=cc^\top$ and $c\in \R^n$).
\end{proposition}

\begin{proof} 
Thanks to Proposition~\ref{prop:hom}, we can restrict to the case $T=1$.
Let $(S_l)_{l\geq 0}\subset \Sy_n^{(PE)}(a,b,1)$ be a minimizing sequence for $R(a,b,1,n)$, i.e., such that there exists a vanishing sequence of positive numbers $(\varepsilon_l)_{l\ge0}$ satisfying $R(S_l)\leq R(a,b,1,n)+\varepsilon_l$ for $l\geq 0$. By the definition of the object at hand, there exists an increasing sequence $(t_l)_{l\geq 0}$ of times tending to infinity and a sequence $(\omega_l)_{l\geq 0}$ of unit vectors such that, for every $l\geq 0$, it holds
\begin{equation}
\ln\Vert\Phi_{S_l}(t_l,0) \omega_l\Vert=\ln\Vert\Phi_{S_l}(t_l,0)\Vert \geq (-R(S_l)-\varepsilon_l)t_l.  
\end{equation}
Fix $l\geq 0$ large. Set 
\begin{equation}
k:=\lfloor t_l\rfloor, \ y_0:=\Phi_{S_l}(t_l-k,0)\omega_l,\
y_{j+1}:=\Phi_{S_l}(t_l-(k-j-1),t_l-(k-j))y_j,\ 0\leq j\leq  k-1.
\end{equation}
From \eqref{eq:exp}, we then get 
\begin{equation}
\ln\left(\frac{\|\Phi_{S_l}(t_l,0)\omega_l\|}{\| y_0\|}\right)=\sum_{j=0}^{k-1}\ln\left(\frac{\Vert y_{j+1}\Vert}{\Vert y_j\Vert}\right)\le -k\mu(a,b,n).
\end{equation}
Clearly, there exists a positive constant $K\leq 1$ independent of $l\geq 0$ such that $K\leq \Vert y_0\Vert\leq 1$.
Thus, since $(t_l)_{l\ge 0}$ is unbounded, we deduce at once that 
\begin{equation}
-R(a,b,1,n)-2\varepsilon_l\leq -k\mu(a,b,n)/t_l.
\end{equation}
By letting $l$ tend to infinity, this yield the l.h.s.\ of \eqref{est-muR}.

The r.h.s.\ of \eqref{est-muR} will follow from the inequality $R(2a,2b,2,n)\leq {\mu(a,b,n)}$ to be proved next.
Let $S_*=c_*c_*^\top\in \Sy_n^{(PE)}(2a,2b,2)$ be the $2$-periodic control given by Proposition~\ref{prop:main} for $T=1$. It then follows from the latter and \eqref{eq:exp} that
\begin{equation}
 \ln\|\Phi_{S_*}(k,0)\|=\sum_{\ell=1}^k\ln\|\Phi_{S_*}(\ell ,\ell-1)\| = -k\mu(a,b,n), \qquad k\in\N.
\end{equation}
Then, standard arguments yield
\begin{equation}
  R(2a,2b,2,n)\le R(S_*)
  \le {-\lim_{\ell\to+\infty}\frac{\ln \|\Phi_{S_*}(2\ell ,0)\|}{2\ell } = \mu(a,b,n)},
\end{equation}
concluding the proof.
\end{proof}

The following links the $L_2$-gain $\gamma(a,b,n,T)$ with the minimal value of \eqref{ocp0}.

\begin{proposition}\label{th-gamma-mu}
 For every $0<a\leq b$, $T>0$ and integer $n\geq 2$, one has  
  \begin{equation}\label{eq:mu-ga}
\frac T{2\mu(a/2,b/2,n)}\leq \gamma(a,b,n,T)\leq \frac T{1-e^{-\mu(a,b,n)}}.
\end{equation}
\end{proposition}
\begin{proof}
Thanks to Proposition~\ref{prop:hom} it suffices to consider the case $T=1$. We start by establishing the right-hand side inequality of \eqref{eq:mu-ga}.
From the variation of constant formula, for every control $u\in L_2((0,+\infty),\mathbb{R}^n)$, $cc^\top\in \Sy_n^{(PE)}(a,b,1)$ and $t\geq 0$, the solution of  \eqref{sys1} with $x_u(0)=0$ reads
\begin{equation}\label{eq:var-const}
x_u(t)=\int_0^t\Phi_{cc^\top}(t,s)u(s)ds,\qquad t\geq 0.
\end{equation}
Since it is easy to deduce from the definition of $\mu:=\mu(a,b,n)$ that $\Vert \Phi_{cc^\top}(t,s)\Vert\leq e^{-\mu \lfloor{t-s}\rfloor}$ for every  $t\geq s\geq 0$, the above implies
\begin{equation}\label{eq:th3-1}
\Vert x_u(t)\Vert\leq \int_0^te^{-\mu \lfloor t-s\rfloor}\Vert u(s)\Vert ds,\qquad t\geq 0.
\end{equation}
Let $h$ be the characteristic function of $\mathbb{R}_+$. Define on $\mathbb{R}$ the function $f(s)=e^{-\mu \lfloor{s}\rfloor}h(s)$, which is  square integrable over $\mathbb{R}_+$. 
Then the r.h.s.\ of \eqref{eq:th3-1} is equal to the convolution product of $f$ and $\|u(\cdot)\|$.
By convolution and Plancherel theorems, one has that the $L_2$-gain of \eqref{sys1} is upper bounded by $\| F\|_\infty$ where $F$ is the Fourier transform of $f$. It is now straightforward to observe that the supremum of $F$ is attained at $0$, %$\omega=0$, 
which yields the desired upper bound. 

We next give an argument for the the left-hand side inequality of \eqref{eq:mu-ga}. For that purpose, consider $S_*=c_*c_*^\top\in \Sy_n^{(PE)}(2a,2b,2)$ and $\omega_*\in S^{n-1}$ as provided by Proposition~\ref{prop:main} for $T=1$. For $t\in [0,2]$, set
\begin{equation}\label{eq:L2-rho}
\rho(t):=\Vert \Phi_{S_*}(t,0)\omega_*\Vert,\qquad \hat{\rho}:=\rho(2)=\exp(-2\mu)<1.
\end{equation}
Since $S_*$ is $2$-periodic, one has that for every $t\geq s\geq 0$ and integers $k\geq l$,
\begin{equation}\label{eq:L2-per}
\Phi_{S_*}(t+2l,s+2l)=\Phi_{S_*}(t,s),\qquad
\Phi_{S_*}(2k,2l)\omega_*=\hat{\rho}^{k-l}\omega_*.
\end{equation}
For $t\geq 0$, set $k_t:=\lfloor t/2\rfloor$ and $\xi_t=t-2k_t$, i.e., $t=2k_t+\xi_t$ with $\xi_t\in [0,2)$. 

For any square-integrable function $v$ defined on $[0,2]$, consider the input function $u:\mathbb{R}_+\to \R^n$ given by
\begin{equation}\label{eq:L2-u}
u(t)=v(\xi_t)\Phi_{S_*}(\xi_t,0)\omega_*,\qquad t\geq 0.
\end{equation}
Observe that $u$ is $2$-periodic. Let $x_u$ be the trajectory of \eqref{sys1} associated with $u$ and starting at the origin. Then, by using \eqref{eq:var-const} and \eqref{eq:L2-per}, one has, for $t\geq 0$,
\begin{equation}
\begin{split}
x_u(t)&=\int_{2k_t}^t\Phi_{S_*}(t,s)u(s)ds+\Phi_{S_*}(t,2k_t)\sum_{j=0}^{k_t-1}\int_{2j}^{2(j+1)}\Phi_{S_*}(2k_t,s)u(s)ds\\
&=\int_{0}^{\xi_t}\Phi_{S_*}(\xi_t,s)u(s)ds+\Phi_{S_*}(\xi_t,0)
\sum_{j=0}^{k_t-1}\int_0^2\Phi_{S_*}(2(k_t-j),s)u(s)ds.
\end{split}
\end{equation}
Set $V(t):=\int_0^tv(s)ds$ for $t\in [0,2]$. Thanks to  \eqref{eq:L2-per} and \eqref{eq:L2-u}, the above yields 
\begin{equation}\label{eq:L2-0}
x_u(t)=\Big(V(\xi_t)+\frac{\hat{\rho}}{1-\hat{\rho}}(1-\hat{\rho}^{k_t})V(2)\Big)
\Phi_{S_*}(\xi_t,0)\omega_*, \qquad t\ge 0.
\end{equation}
Thus, for every positive integer $k$ we have
\begin{equation}
\int_0^{2k}\Vert x_u(t)\Vert^2dt=k\int_0^{2}z^2(t)\rho^2(t)dt+r_k,%O\left(\sum_{j=0}^{k-1}\hat{\rho}^j\right),
\end{equation}
where $z$ is the function defined by 
 \begin{equation}\label{eq:L2-1}
 z(t)=V(t)+\frac{\hat{\rho}}{1-\hat{\rho}}V(2), \qquad t\in [0,2],
 \end{equation}
and $|r_k|\leq C\sum_{j=0}^{k-1}\hat{\rho}^j\leq C/(1-\hat{\rho})$ for some positive constant $C$.
On the other hand,
\begin{equation}
\int_0^{2k}\Vert u(t)\Vert^2dt=k\int_0^{2}v^2(t)\rho^2(t)dt.
\end{equation}

For every positive integer $k$, let $u^{k}$ be the input function defined as 
follows: it is equal to $u$ on $[0,2k]$ and zero elsewhere. We use $x^k$
to denote the trajectory of \eqref{sys1} associated with $u^k$ and starting at the origin. Note that $x^k(t)=\Phi_{c_*}(t-2k,0)x_u(2k)$, for $t\geq 2k$, which decreases exponentially to zero as $t$ tends to infinity.
Then, we have
\begin{equation}\label{eq:L2-2}
\gamma(2a,2b,n,2)\ge \limsup_{k\to\infty}\frac{\Vert x^k\Vert_{L_2}}{\Vert u^k\Vert_{L_2}}= \sqrt{
\frac{\int_0^{2}z^2(t)\rho^2(t)dt}{\int_0^{2}v^2(t)\rho^2(t)dt}}, \qquad \forall v\not\equiv 0.
\end{equation}
By Proposition~\ref{prop:hom} we have $\gamma(2a,2b,n,2)=2\gamma(2a,2b,n,1)$.
Hence, using the upper bound of $\mu$ given in Proposition~\ref{prop:main}, in order to complete the proof of Theorem~\ref{th2}, it suffices to show that there exists $v\not\equiv 0$ such that $z(t) = v(t)/\mu = V'(t)/\mu$ for all $t\ge 0$. By definition of $z$, such a function $v$ exists if and only if there exists $C\neq 0$ such that the nonzero solution of the equation %the following equation admits a non-zero solution,
\begin{equation}
\frac{1}{\mu}V'(t) = V(t) + C,\qquad V(0)=0
\end{equation}
satisfies $\frac{\hat{\rho}}{1-\hat{\rho}}V(2) = C$.
By taking into account \eqref{eq:L2-rho}, it is easy to show that this is the case.
\end{proof}

As a consequence of the previous result, of Proposition~\ref{prop:main}, and of Theorem~\ref{th1}, we now prove Theorem~\ref{th2}.

\begin{proof}[Proof of Theorem~\ref{th2}]
The left-hand side of~\eqref{eq:L2G} is a consequence of the left-hand side of~\eqref{eq:mu-ga} together with~\eqref{eq:mu2}.
Regarding the proof of the right-hand side of~\eqref{eq:L2G}, we first notice that \eqref{eq:mu2} implies that $\mu(a,b,n)\le C_0/2$.
As a consequence of the monotonicity of $x\mapsto \frac{x}{1-e^{-x}}$  we then get
\begin{equation*}
    \frac{1}{1-e^{-\mu(a,b,n)}}=\frac{\mu(a,b,n)}{1-e^{-\mu(a,b,n)}}\frac{1}{\mu(a,b,n)}\le \frac{C_1}{\mu(a,b,n)}, \quad\text{where} \quad C_1= \frac{C_0}{2(1-e^{-C_0/2})}.
\end{equation*}
By using the right-hand side of~\eqref{est-muR} and \eqref{est0} we thus obtain
\begin{equation}
\gamma(a,b,n,T) \le  \frac{TC_1}{\mu(a,b,n)}
\le   \frac{2C_1}{R(2a,2b,T,n)}
\le \frac{2C_1}C\frac{(1+nb^2)T}{a},
\end{equation}
concluding the proof.
\end{proof}

We finally prove Theorem~\ref{th3}, relying on the validity of Proposition~\ref{prop:main}. 

\begin{proof}[Proof of Theorem~\ref{th3}] 
It is enough to prove that the condition provided in the statement of the theorem is a necessary condition for (GAS). Consider the three sequences $(a_l)_{l\geq 1}$, $(b_l)_{l\geq 1}$ and $(\tau_l)_{l\geq 1}$ verifying the assumptions of the theorem. For every $l\geq 1$, we define $T_l:=\tau_{l+1}-\tau_l$ and apply Proposition~\ref{prop:main} to $(a_l,b_l,T_l)$
to deduce that
there exists $S_l$ in $\Sy_n^+(a_l,b_l,T_l)$ and $\omega_l\in \mathbb{S}^{n-1}$
such that $J(S_l,\omega_l)=\mu(a_l,b_l,T_l)$ and the trajectory of \eqref{sys-om} starting at $\omega_l$ and corresponding to $S_l$ is $2T$-periodic.

Choose a sequence $(U_l)_{l\geq 0}$ in $\U(n)$ such that, if 
$\Sigma_l$ is the function defined on $[0,\tau_l]$ as the concatenation of the $U_jS_j$, $0\leq j\leq l-1$ and if $(y_j)_{0\leq j\leq l-1}$ is the sequence defined by $y_0:=\omega_0$ and $y_{j+1}:=\Phi_{\Sigma_l}(\tau_{j+1},\tau_j)y_j$, then one has, for $0\leq j\leq l-1$, that 
\begin{equation}
\frac{\Vert y_{j+1} \Vert}{\Vert y_j\Vert}=\Vert \Phi_{S_j}(T_j,0)w_0\Vert.
\end{equation}
By summing up these relations and using the definitions of the objects at hand, one obtains 
\begin{equation}
-\ln\Vert\Phi_{\Sigma_l}(\tau_l,0)\omega_0\Vert=
\sum_{j=0}^{l-1}\mu(a_j,b_j,n).
\end{equation}
From Proposition~\ref{prop:main}, one deduces  that the series of general term $\mu(a_l,b_l,n)$ converges if and only if the series of general term $\frac{a_l}{1+b_l^2}$ converges. Together with the above equation, one easily concludes.
\end{proof}

\section{Existence of rank one periodic minimizers for \eqref{ocp0}}

In this section, we prove the second part of Proposition~\ref{prop:main}. 
This is done via the following.

\begin{proposition}\label{prop:rank1}
	There exists a rank-one $S_*=c_*c_*^\top\in \Sy_n^{(PE)}(a,b,T)$ 
	and an initial condition $\omega_0\in \mathbb S^{n-1}$ such that 
	\begin{equation}
	\omega_*(t)= \frac{\Phi_{c_*}(t,0)\omega_0}{\|\Phi_{c_*}(t,0)\omega_0\|},
	\end{equation}
	is $2T$-periodic and both $t\mapsto S_*|_{[0,T]}(t)$ and $t\mapsto S_*|_{[T,2T]}(T-t)$, together with their respective initial conditions $\omega_0$ and $\omega_*(T)$, are minimisers for \eqref{ocp0}.
\end{proposition}

In order to prove the above, we apply the Pontryagin Maximum Principle (PMP for short) to the minimizer with constant trace of \eqref{ocp0} given by Proposition~\ref{prop:min}. 
 
As usual, in order to get rid of the constraint~\eqref{PE00} we introduce an auxiliary variable $Q\in \Sy_n$, and reformulate \eqref{ocp0} as follows:
Minimize $J(S,\omega_0)$ with respect to $S\in \Sy_{n}^+(a,b,T)$ and $\omega_0\in\mathbb S^{n-1}$ along trajectories of 
\begin{align}
\dot \om&=-S\om+\left(\om^TS\om\right)\om,\\ %\label{eq:dotomega}\\
\dot Q&=S,
\end{align}
starting at $(\om_0,0)$, and so that  $Q(T)\in \Sy_n(a,b)$. 
The state space of the system is $\mathcal M = \mathbb S^{n-1}\times \Sy_n$. We will henceforth identify the cotangent space at $(\omega,Q)\in \mathcal M$ with $T_\omega^*\mathbb S^{n-1}\times T^*_Q\Sy_n \simeq (\mathbb{R}\omega)^{\perp} \times \Sy_n$.

According to the PMP, 
a solution $(\omega,Q)$ of the optimal control problem~\eqref{ocp0} is necessarily the projection of an \emph{extremal}, i.e., an integral curve $\lambda\in T^*\mathcal M$ of the Hamiltonian vector on $T^*\mathcal M$ satisfying certain additional conditions. We hereby present a definition of extremal adapted to our setting. The fact that this is equivalent to the standard definition of extremal is the subject of the subsequent proposition.

\begin{definition}\label{def:ext}
  A curve $\lambda : [0,T]\to T^*\mathcal M$ is an extremal with respect to\ the control $S\in \Sy^+_n(a,b,T)$ and $\omega_0\in\mathbb S^{n-1}$ if:
  \begin{enumerate}
    \item[(i)]  letting $\lambda = (\omega, Q, p, P_Q)$, it satisfies
  \begin{eqnarray}
    \dot \om&=&-S\om+\left(\om^TS\om\right)\om, \label{eq:dotomega}\\
    \dot Q&=&S,\\
    \dot p &=& Sp -(\omega^\top S \omega)p-\dot \omega, \label{eq:dotp}\\
    \dot P_Q &=& 0.\label{eq:dotPp}
  \end{eqnarray}
    \item[(ii)] It holds that $p(0)=p(T)=0$ and that $-P_Q$ belongs to the normal cone of $\Sy_n(a,b)$ at $Q(T)$.
    \item[(iii)] Let 
    \begin{equation}\label{eq:M}
      M := P_Q -(\omega p^\top +p\omega^\top+\omega\omega^\top) \quad \text{on }[0,T].
       \end{equation}
 Then, $M\le 0$ and $MS = SM \equiv 0$ on $[0,T]$.
\end{enumerate}
\end{definition}

Note that, by conditions $(i),(ii)$ in Definition~\ref{def:ext}, $p^T\omega\equiv 0$ along extremals, which is consistent with the identification $T_\omega^*\mathbb S^{n-1}\simeq (\mathbb{R}\omega)^{\perp}$. 
We then get the following.

\begin{proposition}\label{prop:pmp}
Let $(\omega,Q):[0,T]\to \mathcal M$ be an optimal trajectory of 
  the optimal control problem \eqref{ocp0}, whose optimal control $S$ has constant trace. Then $(\omega,Q)$ is the projection on $\mathcal M$ of an extremal $\lambda:[0,T]\to T^*\mathcal M$.
\end{proposition}

\begin{proof}
Recall that the existence of $(\omega,Q):[0,T]\to \mathcal M$ as an optimal trajectory associated with a control $S$ of constant trace is guaranteed by Proposition~\ref{prop:min}.
After some computations, deferred to Proposition~\ref{prop:ham} in Appendix~\ref{a:ham}, the PMP implies that there exists a curve $t\in[0,T]\mapsto (p(t),P_Q(t))$ and $\nu_0\in\{0,1\}$ with $(p(t),P_Q(t),\nu_0)\neq 0$ a.e.\ on $[0,T]$ such that
\begin{enumerate}
  \item  $(p(t),P_Q(t))\in T_{\omega(t)}^*\mathbb S^{n-1}\times T^*_{Q(t)}\Sy_n$ satisfy on $[0,T]$ the adjoint equations:
    \begin{eqnarray}\label{eq:adj1}
      \dot p &=& Sp -(\omega^\top S \omega)p-\nu_0\dot \omega,\label{eq:p-0} \\
      \label{eq:adj2}
      \dot P_Q &=& 0;
    \end{eqnarray}
  \item letting $\lambda(t) = (\omega(t), Q(t), p(t), P_Q)$ we have the maximality condition:
  \begin{equation}\label{eq:max}
    H(\lambda(t), S(t)) = \max_{S\in\Sy_n^+} H(\lambda(t),S) \qquad\text{a.e.\ on }[0,T],
  \end{equation}
  where $H(\om,Q,p,P_Q,S)= {\Tr( S \tilde M)}/{2}$ is the Hamiltonian of the system and $\tilde M\in \Sy_n$ is defined by 
\begin{equation}\label{ham0}
 %-p^TS\om+\frac{\Tr(SP_Q)}2-\frac{\lambda_0}2\om^TS\om,
\tilde M =  P_Q - \left(\omega p^\top + p\omega^\top + \nu_0\omega\omega^\top\right), \quad \text{with } \nu_0\in\{0,1\}.
\end{equation}
  \item we have the transversality conditions:
  \begin{equation}\label{eq:trans}
    p(0) \perp T^*_{\omega(0)}\mathbb S^{n-1}, \qquad 
    p(T) \perp T^*_{\omega(T)}\mathbb S^{n-1},
  \end{equation}
  and $-P_Q$ belongs to the normal cone of $\Sy_n(a,b)$ at $Q(T)$.
\end{enumerate}

Note that Item \emph{(ii)} of Definition~\ref{def:ext} is equivalent to the transversality conditions, since $p\in T_\omega^*\mathbb S^{n-1}$ by definition. We are left to prove Items \emph{(i)} and \emph{(iii)}.
For this purpose, we claim that the maximality condition implies that
\begin{equation}\label{eq:ham-0}
  H(\lambda(t),S(t)) \equiv 0 \quad \text{and}\quad \tilde M\le 0 \quad\text{on } [0,T].
\end{equation}
Indeed, $H(\lambda(t),0)=0$ and if there exists $\bar S\in\Sy_n^+$ such that $H(\lambda(t),\bar S)>0$ the maximum in \eqref{eq:max} would be infinite, since $H(\lambda(t),\gamma\bar S)\rightarrow+\infty$ as $\gamma\to +\infty$, proving the first part of the claim. As a consequence, $\Tr(S\tilde M) \le 0$ for every $S\in\Sy_n^+$. In particular, for any $z\in\mathbb R^{n}$, $\Tr(zz^\top \tilde M)=z^\top \tilde M z\leq 0$, which shows that $\tilde M \le 0$. Hence, the claim is proved, and the fact that $\Tr(S\tilde M) \le 0$ along optimal trajectories implies $S\tilde M = \tilde M S = 0$. 
Let us now prove that $\nu_0=1$, which will yield at once Items \emph{(i)} and \emph{(iii)}. We argue by contradiction and assume $\nu_0=0$. In this case, \eqref{eq:p-0} is a linear ODE and, due to Item \emph{(ii)}, its solution is $p\equiv 0$. This and \eqref{eq:ham-0} imply that $P_Q\le 0$ and $\Tr(P_QS)\equiv 0$.
Hence, $P_QS\equiv 0$. Integrating over $[0,T]$ this relation yields $P_Q Q(T)=0$, which implies $P_Q=0$ since $Q(T)\ge a\idty$ and hence is invertible. This, however, contradicts the fact that $(p,P_Q,\nu_0)\neq 0$, thus showing that $\nu_0=1$.
\end{proof}

We will also need the following.

\begin{proposition}\label{prop:diag}
  Let $\lambda = (\omega, Q, p, P_Q)$ be an extremal with respect to\ an optimal control $S$. Then, up to an orthonormal change of basis, there exists $k,r\in\N$, with $n=1+k+r$, $\alpha\in(0,1]$, and positive definite diagonal matrices $D_Q\in \R^{r\times r}$ and $D_b\in \R^{k\times k}$, with all elements of $D_Q$ belonging to the interval $[a, b]$,  such that
  \begin{equation}\label{eq:diag}
    Q(T) = \diag(a, b\idty_k, D_Q)  
    \qquad\text{and}\qquad 
    P_Q = \diag(\alpha,-D_b, 0_r).
  \end{equation}
\end{proposition}

\begin{proof}
    Since $P_Q\le \omega_0\omega^\top_0$,  one deduces at once that $\lmax(P_Q)\le \lmax(\omega_0\omega_0^\top)=1$.
    We now claim that $P_Q$ has exactly one positive eigenvalue $\alpha\in (0,1]$. 
 
 Let us first show that $P_Q$ has at most one positive eigenvalue. Indeed, by Item \emph{(iii)} of Definition~\ref{def:ext}, $P_Q-\omega_0\omega^\top_0$ is negative semi-definite. Therefore, the restriction of the quadratic form defined by $P_Q$ to $(\mathbb{R}\omega_0)^\perp$ is also negative semi-definite. This implies that $P_Q$ has at least $n-1$ non positive eigenvalues.

We next show that $P_Q$ cannot be negative semi-definite. Arguing by contradiction, one has that for every $t_1\leq t_2$ in $[0,T]$,  
\begin{equation}\label{eq:tr0}
\Tr\Big(P_Q\big(Q(t_2)-Q(t_1)\big)\Big)\leq 0.
\end{equation}
Let $T_0\le T$ be the largest time in $[0,T]$ such that $S\omega\equiv 0$ and $p\equiv 0$ on $[0,T_0]$.
We first prove that $T_0$ exists and is strictly positive. For that purpose, pick $\bar t\in [0,T]$ such that $\|p(\bar t)\|=\max\{\|p(s)\|\mid 0\le s\le \bar t\}$ and $\Tr(Q(\bar t))\le 1/2$. One deduces that
\begin{equation}\label{eq:wow}
\int_0^{\bar t} p^\top Sp\,dt \le \max_{s\in[0,\bar t]}\|p(s)\|^2
\int_0^{\bar t} \lambda_{\max}(S)\,dt\leq
\|p(\bar t)\|^2\Tr(Q(\bar t)) \le 
\frac{\|p(\bar t)\|^2}{2}.
\end{equation}
We next prove the following two equalities, holding for every 
$t_1\leq t_2$ in $[0,T]$,
\begin{multline}\label{ipp0}
\frac12\big(\Vert p(t_2)\Vert^2-\Vert p(t_1)\Vert^2\big)+\int_{t_1}^{t_2}(\om^T(t)S(t)\om(t))\Vert p(t)\Vert^2dt\nonumber\\
 =\int_{t_1}^{t_2}p^T(t)S(t)p(t)dt+\int_{t_1}^{t_2}p^T(t)S(t)\om^T(t)dt,
\end{multline}
and
\begin{equation}\label{ipp1}
\Tr\Big(P_Q\big(Q(t_2)-Q(t_1)\big)\Big)=2\int_{t_1}^{t_2}p^T(t)S(t)\om^T(t)dt+\int_{t_1}^{t_2}\om^T(t)S(t)\om(t)dt.
\end{equation}	
Both equalities follow by Proposition~\ref{prop:pmp}: for the first one, we multiply by $p^\top$ the dynamics of $\dot p$ given by \eqref{eq:dotp} and integrate it on $[t_1,t_2]$ using the fact that $p^\top\omega = 0$. %, and then use the transversality conditions of Item \emph{(ii)}. 
We integrate $\Tr(MS)$ over $[t_1,t_2]$, with $M$ given in~\eqref{eq:M},  to obtain the second one.

It is immediate to deduce from \eqref{ipp0} and \eqref{ipp1} that, 
for every  $t_1\leq t_2$ in $[0,T]$,
\begin{multline}\label{ipp2}
\int_{t_1}^{t_2}p^T(t)S(t)p(t)dt=\frac12\int_{t_1}^{t_2}\om^T(t)S(t)\om(t)dt
+\frac12\big(\Vert p(t_2)\Vert^2-\Vert p(t_1)\Vert^2\big)\\
+\int_{t_1}^{t_2}(\om^T(t)S(t)\om(t))\Vert p(t)\Vert^2dt
-\frac12\Tr\Big(P_Q\big(Q(t_2)-Q(t_1)\big)\Big).	
\end{multline}
By using \eqref{eq:tr0},\eqref{eq:wow} and \eqref{ipp2} with $t_1=0$ and $t_2=\bar t$, one deduces that 
\begin{equation}
\int_0^{\bar t}\omega^\top S\omega \,dt \le \Tr(P_QQ(\bar t))\le 0.
\end{equation}
This immediately implies that $S\omega \equiv 0$, $p \equiv 0$ on $[0,\bar t]$, proving the existence of $T_0$ as claimed.	
Note that, necessarily $T_0<T$, since otherwise one would have that $\omega\equiv \omega_0$ on $[0,T]$ and, integrating $S\omega\equiv 0$ on $[0,T]$ would yield that $Q(T)\omega_0=0$, contradicting the fact that $Q(T)\ge a\idty$. 

We next pick $T_0+\bar t\in [0,T]$ such that $\|p(T_0+\bar t)\|=\max\{\|p(s)\|\mid T_0\le s\le T_0+\bar t\}$ and $\Tr(Q(T_0+\bar t))-\Tr(Q(T_0))\le 1/2$. We then reproduce the argument starting in \eqref{eq:wow} where we replace the pair of times $(0,\bar t)$ by the pair of times $(T_0,T_0+\bar t)$.
In that way, we extend the interval on which both $S\om$ and $p$ are zero beyond $T_0$, hence contradicting the definition of $T_0$. We have completed the argument for the existence of a unique positive eigenvalue $\alpha$ for $P_Q$. 

	We are left to show that $P_Q$ and $Q(T)$ can be put in the form \eqref{eq:diag} by an orthonormal change of basis.
	By definition $-P_Q$ belongs to the normal cone of $\Sy_n(a,b)$ at $Q(T)$ if and only if $\Tr(P_Q (\hat Q-Q(T)))\geq 0$ for any $\hat Q \in\Sy_n(a,b)$. 
	Assume without loss of generality that $Q(T)$ is diagonal and let $\lambda_i$, $i=1,\dots,n$, be the eigenvalue of $Q(T)$ corresponding to the eigenvector $e_i$ of the canonical basis. 
	If $\lambda_i\in (a,b)$, then  it is easy to check that the matrices $\hat Q_{\pm} = Q(T) \pm \varepsilon (e_ie_j^T+e_je_i^T)$, for $j=1,\dots,n$, belong to $\Sy_n(a,b)$ if $\varepsilon>0$ is small enough. Using the fact that $\Tr(P_Q (\hat Q_+ -Q(T)))\geq 0$ and $\Tr(P_Q (\hat Q_- -Q(T)))\geq 0$ one gets that the $(i,j)$ component of $P_Q$ must be $0$. If $\lambda_i=a$ then $\Tr(P_Q (\hat Q-Q(T)))\geq 0$, with $\hat Q = Q(T) + \varepsilon e_ie_i^T\in \Sy_n(a,b)$, implies that the component $(i,i)$ of $P_Q$ is nonpositive. Similarly one deduces that $\lambda_i=b$ implies that the component $(i,i)$ of $P_Q$ is nonnegative.
	Consider now any two eigenvalues $a\leq \lambda_i<\lambda_j\leq b$ of $Q(T)$. Then it is easy to check that the matrices $\hat Q_{\pm} = Q(T) \pm \varepsilon (e_ie_j^T+e_je_i^T)+k\varepsilon^2 (e_ie_i^T-e_je_j^T)$  with $k>1/(\lambda_j-\lambda_i)$ belong to $\Sy_n(a,b)$ if $\varepsilon$ is small enough. Again, since $\Tr(P_Q (\hat Q_{\pm}-Q(T)))\geq 0$, and letting $\varepsilon$ tend to zero, one gets that the $(i,j)$ component of $P_Q$ must be zero. One deduces that $P_Q$ commutes with $Q(T)$ and the two matrices can thus be simultaneously diagonalized, taking the form~\eqref{eq:diag}.
\end{proof}
	  
\begin{proposition}\label{prop:dim-red}
	Let $\lambda=(\omega, Q,p,P_Q)$ be an extremal of \eqref{ocp0} associated with a control $S$.
  Assume that, in the notations of Proposition~\ref{prop:diag}, one has that $r\ge 1$. 
  %Then, the optimal control reads $S = \diag(\tilde S, S_0)$, where $\tilde S \in\Sy^+_{k+1}(a,b,T)$ and $S_0 \in\Sy^+_{r}(a,b,T)$. 
  Then, there exist $(\tilde \omega, \tilde p)\in T^*\mathbb{S}^{k}$, $\tilde S \in\Sy^+_{k+1}(a,b,T)$ and $S_0 \in\Sy^+_{r}(a,b,T)$, such that 
	\begin{equation}
	\omega = (\tilde \omega,0)^\top,
	\qquad	
	p = (\tilde p,0)^\top
	\qquad\text{and}\qquad
	S = \diag(\tilde S, S_0).
\end{equation}	  

  Moreover, letting $\tilde Q = \diag(a,b\idty_k)$ and $\tilde P_Q = \diag(\alpha,-D_b)$, we have that $\tilde \lambda = (\tilde \omega, \tilde Q, \tilde p, \tilde P_Q)$ is an extremal trajectory with control $\tilde S$ of \eqref{ocp0} in dimension $k+1$, and $J(\tilde S, \tilde\omega(0))=J(S,\omega(0))$. In particular, if $k=0$, then $p\equiv 0$ and there exists $\omega_0\in \mathbb{S}^{n-1}$ such that  $\omega\equiv\omega_0$ and $J(S,\omega_0)=a$.
\end{proposition}
	  
\begin{proof}
	We start by decomposing $\omega = (\tilde \omega,\xi)$ and $p=(\tilde p,q)$ for some $\R^{r}$-valued functions $\xi$ and $q$. Our aim is to prove that $q\equiv\xi\equiv0$. Let us define
	\begin{gather}
	  \tilde A = (\tilde p + \tilde \omega)(\tilde p + \tilde \omega)^\top -\tilde p \tilde p^\top, 
	  \qquad
	  A_0 = ( q + \xi)(q + \xi)^\top - q q^\top
	  \\
	  B = (\tilde p + \tilde \omega)(q + \xi)^\top -\tilde p q^\top
	\end{gather} 
	Then, by Item (\emph{iii.}) of Definition~\ref{def:ext}, %Proposition~\ref{prop:pmp},
	we get
	\begin{equation}\label{eq:R}
	 \left(
	  \begin{array}{cc}
	     \tilde A - \tilde P_Q & B \\ B^\top & A_0
	  \end{array}
	  \right)
      \ge0.
	\end{equation}
	We deduce at once that $A_0\ge 0$, and thus, that there exists $\varrho\in[-1,1]$ such that $q=\varrho(q+\xi)$. In particular, $q+\xi=0$ if and only if $q=\xi=0$. 
	Let $I$ be a maximal open interval such that $q+\xi\neq 0$ and assume, by contradiction, that $I\neq \varnothing$.
	
	We claim that
		\begin{equation}\label{eq:Y}
	  - (1-\varrho^2)\tilde P_Q \ge \big((1-\varrho)\tilde p+\varrho\tilde \omega\big)\big((1-\varrho)\tilde p+\varrho\tilde \omega\big)^\top 
	   \qquad \text{on } I.
	\end{equation}
  To this effect, set $A_\varepsilon = A_0 + \varepsilon(q+\xi)(q+\xi)^\top$ for $\varepsilon>0$. Observe that \eqref{eq:R} holds with $A_0$ replaced by $A_\varepsilon$. Then, by Schur complement formula we have
	\begin{equation}\label{eq:schur}
	  \tilde A- \tilde P_Q- BA_\varepsilon^\dagger B^\top \ge 0,
	\end{equation}
	where we denoted by $A_\varepsilon^\dagger$ the Moore-Penrose inverse of $A_\varepsilon$. 
	Let us observe that
	\begin{equation}
	  A_\varepsilon = (1-\varrho^2+\varepsilon)( q + \xi)(q + \xi)^\top
	  \quad\text{and}\quad
	  B = \left((1-\varrho)\tilde p + \tilde \omega\right)(q + \xi)^\top.
	\end{equation} 	
    Since $A_0^\dagger = \frac{(q+\xi)(q+\xi)^\top}{(1-\varrho^2+\varepsilon)\|q+\xi\|^4}$, the claim follows by letting $\varepsilon\downarrow 0$ in \eqref{eq:schur} and simple computations. %The case $\varrho^2=1$ can then be recovered by a simple continuity argument.
		
	In order to obtain the desired contradiction, we observe that it has to hold $\varrho^2=1$. Indeed, by \eqref{eq:Y}, we have $-(1-\varrho^2)\alpha\ge 0$ with $\alpha>0$. On the other hand, if $\varrho = 1$, we have $\xi\equiv0$ on $I$ by definition of $\varrho$, and $\tilde \omega\equiv 0$ on $I$ by \eqref{eq:Y}, which contradicts $\omega\in \mathbb S^{n-1}$. Thus, $\varrho=-1$ and thus, by \eqref{eq:Y}, it holds $2p+\omega\equiv 0$ on $I$. However, since $p\in\omega^\perp$, we have $\|2p+\omega\|\ge \|\omega\|\equiv 1$, thus yielding the desired contradiction. This implies that $I=\varnothing$, and thus that $\xi\equiv q\equiv 0$ on $[0,T]$.	
	
	 Setting $S =  \begin{pmatrix} \tilde S & S_D \\ S_D^\top & S_0 \end{pmatrix}$, it is  easy to check from~\eqref{eq:dotomega} and~\eqref{eq:dotp} that $\tilde{\omega}^\top S_D  \equiv  0 $ and $\tilde{p}^\top S_D \equiv 0 $. Then, it follows from Item~$(iii)$ of Definition~\ref{def:ext} that $\tilde{P}_QS_D\equiv 0$, so that we can conclude that $S_D\equiv 0$. This yields the desired form for $S$, together with the fact that $\tilde \lambda$ is an extremal trajectory with control $\tilde S$ of \eqref{ocp0} in dimension $k+1$ satisfying $J(\tilde S, \tilde\omega(0))=J(S,\omega(0))$. Finally, the last part of the statement follows by the explicit computation of the solutions in dimension $n=1$.
\end{proof}

As we will see, the above Proposition immediately yields Proposition~\ref{prop:rank1} if we are in the case $k=0$. Thus, we henceforth focus on extremals that satisfy the following. 

\begin{assumption}
	\label{ass:extremal}
	The extremal $\lambda = (\omega,Q,p,P_Q)$ is such that 
	\begin{equation}
		Q(T) = \diag(a, b\idty_{n-1}) 
		\qquad\text{and}\qquad
		P_Q = \diag(\alpha,-D_b),
	\end{equation}
	where $\alpha\in(0,1]$, and $D_b\in\R^{(n-1)\times(n-1)}$ is a positive definite diagonal matrix.
\end{assumption}

We start by proving some essential properties of the matrix 
\begin{equation}
	M = P_Q - \left(\omega p^\top + p \omega^\top + \omega\omega^\top\right).
\end{equation}
We recall that, by Item (\emph{iii.}) of Definition~\ref{def:ext}, %Proposition~\ref{prop:pmp}, 
we have
\begin{equation}\label{eq:prop-M}
M\le 0
 \qquad\text{and}\qquad
 MS\equiv SM \equiv 0.
\end{equation}

\begin{proposition}\label{prop:prop-M}
	Let $\lambda$ be an extremal satisfying Assumption~\ref{ass:extremal}. Then, $\rank M\equiv n-1$, and $M$ has constant spectrum (taking into account multiplicities). In particular, $\rank S \equiv 1$.
\end{proposition}

\begin{proof}
	By \eqref{eq:M}, we have $\rank M\le n-1$. Let $x\in \ker M$, $x\neq 0$, then 
  \begin{equation}\label{eq:ker}
    x = \big((p+\omega)^\top x\big)P_Q^{-1}\omega+ (\omega^\top x)P_Q^{-1}p.
  \end{equation}
  As a consequence, it holds $\ker M\subset V:=\operatorname{span}\{P_Q^{-1}\omega, P^{-1}_Qp\}$. Observe that $\dim V \in\{1, 2\}$, with $\dim V=2$ if and only if $p\neq 0$. In particular, $\rank M\ge n-2$. 
  
  Let $E = \{ t\in[0,T]\mid \rank M(t)=n-1 \}$. Trivially, $E$ is open in $[0,T]$, and, moreover, $E\neq \varnothing$, since $\{0,T\}\subset E$. We now show that $E$ is also closed, which implies $E=[0,T]$ thus completing the proof of the statement.

  To this aim, let us start by observing that, thanks to \eqref{eq:prop-M}, on $E$ it holds $S=cc^\top$ for some $c\in \ker M$. 
  By Proposition~\ref{prop:min}, one can assume that $c$ has constant norm.
  We now claim that $t\in E\mapsto c(t)$ is an analytic function of $(\omega(t),p(t))$.
  To see that, set $\alpha = (p+\omega)^\top x$ and $\beta = \omega^\top x$. By multiplying \eqref{eq:ker} by $(p+\omega)^\top$ and $\omega^\top$, one observes that for all $t\in E$ the vector $(\alpha(t),\beta(t))$ belongs to the kernel of a $2\times 2$ matrix whose entries are degree $2$ polynomial functions of the coordinates of $(\omega(t),p(t))$.
  The claim follows by using the fact that $c$ has constant norm.
%   Since, on $E$, an element of $\ker M$ can be expressed as an analytic function of $(\omega(t),p(t))$ by \eqref{eq:ker}, 
  By the dynamics \eqref{eq:dotomega}-\eqref{eq:dotp}, this implies that the extremal trajectory and $t\mapsto M(t)$ are analytic on $E$. As a consequence, the (unordered) negative eigenvalues $(-\lambda_j)_{j=1}^{n-1}$ of $M$ are analytic on $E$, and it is possible to find an analytic family $\{v_1,\ldots,v_{n-1}\}$ of associated orthonormal eigenvectors, (see, e.g., \cite[Theorem~6.1 and Section 6.2]{Kato1995}).
  
  Differentiating with respect to\ $t\in E$ the relation $v^\top_\ell M v_\ell = \lambda_\ell$ for $\ell = 1,\ldots,n-1$, we have
  \begin{equation}\label{eq:dotM}
  	v_\ell ^\top \dot M v_\ell = \dot \lambda_\ell \quad\text{on } E.
  \end{equation}
  To complete the proof of the statement, observe that straightforward computations from the definition of $M$ yield
  \begin{equation}\label{eq:dM}
    \dot M = (\omega^\top c)(cp^\top+pc^\top)-(p^\top c)(c\omega^\top + \omega c^\top) \quad \text{on }E.
  \end{equation}
  Using this and the fact that $c^\top v_\ell\equiv 0$ in \eqref{eq:dotM} yields that $\dot\lambda_\ell\equiv 0$, $ \ell=1,\ldots,n-1$. 
  This yields at once that $E$ is closed, completing the proof of the statement.
\end{proof}

\begin{proposition}\label{prop:quasi-per}
	Let $\lambda$ be an extremal satisfying Assumption~\ref{ass:extremal}.
	Then, 
	\begin{equation}
	\omega_i(0)^2 = \omega_i(T)^2
	\qquad\text{for}\qquad
	i=1,\ldots,n
	\end{equation}	
\end{proposition}

\begin{proof}
	Let us write $P_Q = \diag(-d_1,\dots,-d_{n})$, where $d_1=\alpha$ and $0<d_i\le d_{i+1}$, for $2\le i\le n$.
	By \eqref{eq:prop-M} and Proposition~\ref{prop:prop-M}, the control associated with $\lambda$ takes the form $S = cc^\top$ for some %$c\in \mathcal C_n^0(a,b,T)$, 
	vector valued function $c\not\equiv 0$.
	Set 
    \begin{equation}
 q=p+\omega/2,\ \gamma = c^\top\omega,\ \delta = c^\top q.
    \end{equation}
 Then, the fact that $Mc\equiv 0$ yields
	\begin{equation}\label{eq:c-dyn}
		c_i = -{d_i}^{-1} \left( (c^\top q ) \omega_i + (c^\top \omega)q_i \right), \qquad i=1,\ldots, n.
	\end{equation}
	Moreover, since by Assumption~\ref{ass:extremal} we have $Q(T) = \diag(a,b\idty_{n-1})$, this implies
	\begin{equation}\label{eq:c-int}
	\int_0^T c_1 c_i\,dt = a \delta_{1i}
	\quad\text{and}\quad
	\int_0^T c_i c_j \,dt= b \delta_{ij}, 
	\quad
	i,j=2,\ldots, n.
	\end{equation}

	Finally, letting $Z_i=(\omega_i,q_i)^\top$ for $i=1,\ldots, n$, we have the following dynamics:
\begin{equation}\label{eq:Z}
\dot Z_i = A_i Z_i, \quad \text{where }
A_i(t) := \frac{1}{d_i}
\left(\begin{array}{cc}
\gamma\left( \delta + d_i \gamma\right) & \gamma^2\\
\delta^2 & -\gamma\left( \delta + d_i \gamma\right)\\
\end{array}
\right).
\end{equation}

	Note that $A_i$ depends on $i$ only through the $d_i$'s. The above implies at once that both $\omega_i(0), \omega_i(T)$ are non zero since, otherwise, from $p_i(0)=p_i(T)=0$, we would have $Z_i(0)$ or  $Z_i(T) = 0$, and thus $Z_i\equiv 0$. However, this would yield $c_i\equiv 0$, contradicting \eqref{eq:c-int}.
	
We now claim that the $d_i$'s are two by two distinct. Indeed, if this were not the case, we would have $d_i=d_{i+1}$ for some $i\ge 2$. By \eqref{eq:Z}, this implies that 
\begin{equation}
	Z_{i+1}=\frac{\omega_{i+1}(0)}{\omega_i(0)} Z_i \quad\text{on }[0,T].
\end{equation}
Then, by \eqref{eq:c-dyn}, we have
\begin{equation}
	c_{i+1}=\frac{\omega_{i+1}(0)}{\omega_i(0)} c_i \quad\text{on }[0,T],
\end{equation}
which yields
\begin{equation}
\int_0^T c_i c_{i+1}\,dt = \frac{\omega_{i+1}(0)}{\omega_i(0)} \int_0^T c_i^2\,dt = b \frac{\omega_{i+1}(0)}{\omega_i(0)} \neq 0.
\end{equation}
However, this contradicts \eqref{eq:c-int} and thus proves the claim. 

Let us now denote by $\fp$ and $\fm$ the characteristic polynomials of the matrices $P_Q$ and $M$, respectively. That is, the degree $n$ polynomials in the indeterminate $\xi$ given by 
\begin{equation}
    \fm(\xi) = \det(\xi \idty_n-M),
    \qquad\text{and}\qquad
    \fp(\xi) = \det(\xi \idty_n-P_Q).
\end{equation}
Observe that, since $P_Q$ and $\spec(M)$ (taking into account multiplicities) are independent of $t\in[0,T]$, the same is true for $\fp$ and $\fm$. In order to complete the proof of the statement, we will compute the ratio $\fm/\fp$ in two different ways.

Firstly, we observe that, by definition of $M$, it holds
\begin{equation}
	\frac{\fm(\xi)}{\fp(\xi)} = \det\left( \idty_n + (\xi \idty_n - P_Q)^{-1}(\omega\omega^\top+p\omega^\top + \omega p^\top) \right)
	=\det(\idty_n + VW^\top),
\end{equation}
where we defined the $n\times 2$ real-valued matrices $V$ and $W$ by
\begin{equation}
V = ((\xi \idty_n - P_Q)^{-1}\omega, (\xi \idty_n - P_Q)^{-1} p)
\qquad\text{and}\qquad
 W = (\omega+p,\omega).
\end{equation}
Then, Sylvester's determinant identity yields
\begin{equation}\label{eq:ratio-1}
	\frac{\fm(\xi)}{\fp(\xi)} 
		=\det(\idty_2 + W^\top  V)
		=1 + \sum_{i=1}^n \left(\omega_i^2 + 2\omega_ip_i + \sum_{j\neq i}\frac{(\omega_ip_j-\omega_jp_i)^2}{d_i-d_j}\right) \frac{1}{\xi+d_i}.
\end{equation}
Here, in the last equality we have used the fact that the $d_i$'s are two by two distinct.

On the other hand, by partial fraction decomposition of $\fm/\fp$ in terms of the indeterminate $\xi$, we have
\begin{equation}
\frac{\fm(\xi)}{\fp(\xi)} 
		=1 + \sum_{i=1}^n \frac{\fm(-d_i)}{\fp'(-d_i)}\frac{1}{\xi+d_i}.
\end{equation}
By comparing the above with \eqref{eq:ratio-1}, we finally obtain the following integrals of motion
\begin{equation}
	\frac{\fm(-d_i)}{\fp'(-d_i)} = \omega_i^2 + 2\omega_ip_i + \sum_{j\neq i}\frac{(\omega_ip_j-\omega_jp_i)^2}{d_i-d_j}, \qquad \text{on }[0,T].
\end{equation}
The statement then follows by evaluating the above at  $t=0$ and $t=T$, and using the transversality conditions  $p(0)=p(T)=0$.
\end{proof}

We are finally in a position to prove the main result of this section.

\begin{proof}[Proof of Proposition~\ref{prop:rank1}]
	Let $\lambda=(\omega,Q,p,P_Q)$ be an extremal of \eqref{ocp0} associated with an optimal control $S$. By Proposition~\ref{prop:diag}, up to an orthonormal change of basis, there exist $\alpha\in(0,1]$, $r\in\N$, and a positive diagonal matrix $D_Q\in \R^{r\times r}$ such that we have the following dichotomy:
	\begin{enumerate}
		\item $r=n-1$, $Q(T)=\diag(a, D_Q)$ and $P_Q=\diag(\alpha, 0_{n-1})$;
		\item $r= n-1-k$ for some $k\in\{1,\dots,n-1\}$, $Q(T)=\diag(a, b\idty_k, D_Q)$ and $P_Q=\diag(\alpha, -D_b, 0_{r})$ for some positive diagonal matrix $D_b\in\R^{k\times k}$;
	\end{enumerate}
	
	In the first case, by Proposition~\ref{prop:dim-red}, we have $J(S, \omega(0))=a$. Then, it suffices to consider a periodic version $S_*\in\Sy_n^{(PE)}(a,b,T)$ of the control $S\in \Sy_n^{+}(a,b,T)$ defined in Proposition~\ref{prop:upper-a}. Indeed, it holds $J(S, \omega(0))=a$, and thus $S$ is an optimal control, and moreover the corresponding trajectory is periodic since $\omega\equiv\omega(0)$.
	
	Let us now focus on the second case. We start by assuming that $r=0$, i.e., that $\lambda$ satisfies Assumption~\ref{ass:extremal}. Then, since $MS\equiv 0$ by \eqref{eq:prop-M} and $\rank M=n-1$ by Proposition~\ref{prop:prop-M} we have $S=cc^\top$ for some vector-valued function $c$. Moreover, by Proposition~\ref{prop:quasi-per}, there exists a diagonal matrix $D$ with entries $\pm1$ such that $\omega(T)=D\omega(0)$ and $c(T)=Dc(0)$. We next define the required $S_*\in \Sy_n^{(PE)}(2a,2b,2T)$ as $S_*=c_*c_*^\top$, where $c_*$ is the $2T$-periodic vector valued function satisfying $c_*(t) = c(t)$  for $t\in [0,T]$ and $c_*(t)=Dc(t-T)$ for $t\in [T,2T]$.
% 	On $[0,T]$, it is equal to $c$ and, for $t\in [T,2T]$, we take $c_*(t):=Dc(t-T)$. 
	Clearly, the corresponding trajectory $\omega_*$ starting at $\om(0)$ will satisfy $\om_*(t)=D\om(t-T)$ for $t\in [T,2T]$ due to the fact that $D\in\U(n)$ and the invariance of the dynamics by elements of $\U(n)$. In particular, $\om_*(2T)=\om_*(0)$ since $D^2=\idty_n$	and similarly $c_*(2T)=c_*(0)$. 

Finally, the case $r\ge 1$ is obtained from the case $r=0$ as follows. 
Observe that the control $ \diag(S_0,0)$ where $S_0\in \Sy_{k+1}^{+}(a,b,T)$ is the rank-one optimal control given by Proposition~\ref{prop:prop-M} with initial condition $\omega(0)$, has cost $J( \diag(S_0,0), \tilde\omega(0)) = \mu(a,b,k+1) = \mu(a,b,n)$ for $\tilde\omega(0)=(\omega(0),0)$. Fix any rank-one $S_1\in \Sy_r^{+}(a,b,T)$. Define the control $S\in \Sy_n^+(a,b,T)$ as follows:
\begin{equation}
    S(t)=
    \begin{cases}
        2\diag(S_0(2t),0),&\quad \text{for }t\in [0,T/2],\\
        2\diag(0,S_1(2t-T)),&\quad \text{for }t\in [T/2,T].\\
    \end{cases}
\end{equation}
This is a rank-one optimal control whose associated trajectory satisfies  $\omega(T)=D\omega(0)$ for some diagonal matrix $D$ with entries $\pm1$. Then, the same procedure used in the case $r=0$ yields the desired $S_*\in \Sy_n^{(PE)}(2a,2b,2T)$.
\end{proof}

\section{The 2D case}

In this section, we prove the first part of Proposition~\ref{prop:main}, which, thanks to Proposition~\ref{prop:mu-n} reduces to the following.

\begin{proposition}\label{prop:2d}
  There exists an universal constant $C_0>0$ such that, for every $0<a\le b$, one has
  \begin{equation}\label{eq:est-fun}
    \mu(a,b,2)\le \frac{C_0 a}{1+b^2}.
  \end{equation}  
\end{proposition}

We start by introducing adapted notations for the 2D case. 
Given a vector $v\in \R^2$ we denote by $v^\perp$ its counter-clockwise rotation of angle $\pi/2$. Moreover, for $\theta\in \mathbb R$ we write $c_\theta$ and $s_\theta$ to denote $\cos\theta$ and $\sin\theta$, respectively.

Observe that $\mathcal M$ is equal to $\mathbb{S}^1\times \Sy_2$ and, thanks to Proposition~\ref{prop:diag}, up to an orthonormal change of basis, we can represent an extremal $\lambda = (\omega,Q,p,P_Q)\in T^*\mathcal M$ as $\lambda = (\theta, Q, \eta, \alpha, d)\in \mathbb R\times \Sy_2\times \R\times (0,1]\times [0,+\infty)$, via the following identifications
\begin{equation}\label{eq:coord}
  \omega = e^{i\theta/2}, \quad p = \eta \omega^\perp = \eta ie^{i\theta/2},\qquad 
  P_Q = \diag(\alpha, -d), \quad Q(T) = \diag(a,b).
\end{equation}

\subsection{Structure of extremals}
In this subsection, we consider a fixed extremal $\lambda= (\theta, Q, \eta, \alpha, d)$ satisfying Assumption~\ref{ass:extremal}, and associated with an optimal control $S=cc^\top$ of constant trace. This implies $d>0$.
The \eqref{PE0} condition and \eqref{eq:coord} then yields, up to a time reparametrization, that
\begin{equation}
  \|c\|=1
  \qquad\text{and}\qquad
  T=a+b.
\end{equation}
Moreover, the matrix $M$ defined in \eqref{eq:M} is non-trivial and can be written on $[0,T]$ as
\begin{equation}\label{eq:M1}
M=\diag(\alpha,-d)-\eta\left(\omega(\omega^{\perp})^\top+\omega^{\perp}\omega^\top\right)-\omega \omega^\top.
\end{equation}
Since $Mc=0$ on $[0,T]$ and the trace of $M$ is constant and equal to $\alpha-d-1$, it holds% one has that, for a.e. $t\in [0,T]$, 
\begin{equation}\label{eq:M2}
M(t)=(\alpha-d-1)c^{\perp}(t)(c^{\perp}(t))^T, \qquad \text{for a.e. } t\in [0,T].
\end{equation}
Since this ensures that $c$ is actually absolutely continuous, this equality holds on the whole interval $[0,T]$.
 
In the following result, we rewrite the dynamics of an extremal trajectory with the adapted notations for the 2D case.

\begin{lemma}\label{lem:dynamics-2d}
  Letting $c=e^{i\phi/2}$, $\phi\in\mathbb R$, we have the following dynamics
  \begin{equation}\label{eq:dynamics-2d}%\tag{$\Sigma$}
  \dot\theta = s_{\theta -\phi} ,
  \qquad
  \dot \eta = -\frac{s_{\theta-\phi}+2\eta c_{\theta-\phi}}2,
  \qquad
  \dot\phi = \frac{2\eta}{1-\alpha+d}.
  \end{equation}
  Moreover, $\eta(0)=\eta(T)=0$ and, for every $t_\star\in[0,T]$ such that $\eta(t_\star)=0$, we have
  \begin{gather}
  \label{eq:theta0-cos}
    c_{\theta(t_\star)} = 1-\frac{2d(1-\alpha)}{\alpha+d},\\
    \label{eq:phi0-cos}
    c_{\phi(t_\star)} = -1+\frac{2d(1+d)}{(\alpha+d)(1-\alpha+d)},
  \end{gather}
  with $s_{\theta(t_\star)}s_{\phi(t_\star)}<0$. 
  In addition, $M(t)c(t)=0$ along trajectories of~\eqref{eq:dynamics-2d} satisfying the previous conditions, for $t\in [0,T]$.
\end{lemma}

\begin{proof}
  The transversality conditions (Item \emph{(ii)} of Definition~\ref{def:ext}) imply immediately that $\eta(0)=\eta(T)=0$.
  The first two equations of \eqref{eq:dynamics-2d} follow at once from \eqref{eq:dotomega}-\eqref{eq:dotp} and \eqref{eq:coord}.
  Let us prove the last one.
  Due to the fact that $2\dot c = \dot\phi c^\perp$ and $2\dot c^\perp = -\dot\phi c$, differentiating \eqref{eq:M2} yields
  \begin{equation}\label{eq:dot-M-psi}
    \dot M = \frac{1-\alpha+d}2 \dot\phi (c^\perp c^\top +c(c^\perp)^\top).
  \end{equation}
  In particular,
  \begin{equation}\label{eq:dot-M-cc}
  c^\top \dot M c^\perp = \frac{1-\alpha+d}2 \dot\phi.
  \end{equation}
  Hence, replacing the expression of $\dot M$ given in \eqref{eq:dM} and using~\eqref{eq:coord}, the left-hand side of~\eqref{eq:dot-M-cc} turns out to be equal to $\eta$. This proves the last equation of \eqref{eq:dynamics-2d}.
  Equations \eqref{eq:theta0-cos} and \eqref{eq:phi0-cos} follow at once by developing the equation $M(t_\star)c(t_\star)=0$ at every $t_\star\in[0,T]$ such that $\eta(t_\star)=0$.
  
We now let $R_{\phi/2}$  be the matrix corresponding to the counter clock-wise rotation by $\phi/2$ and consider the matrix $\tilde M = R_{\phi/2}^\top MR_{\phi/2}$. Direct computations show that
the components $\tilde M_{1,1}$ and $\tilde M_{2,1}$ of $\tilde M$ are given by
\begin{align}
2\tilde M_{1,1}&=-1 + \alpha - d + (\alpha + d) c_\phi - c_{\theta-\phi} + 2 \eta s_{\theta-\phi},\label{M11}\\
2\tilde M_{2,1}&=- (\alpha + d) s_\phi - s_{\theta-\phi} - 2 \eta c_{\theta-\phi}.\label{M21}
\end{align}
  As  the first column of $\tilde M$ is equal to zero, we have that $\tilde M_{1,1}=\tilde M_{2,1}=0$ along the trajectory. 
  In particular one has $2\tilde M_{1,1}c_\theta-2\tilde M_{2,1}s_\theta=(-1 + \alpha - d)c_\theta+ (\alpha + d) c_\phi  c_\theta+(\alpha + d) s_\phi s_\theta-c_{\phi}-2 \eta s_{\phi}=0$ so that, from \eqref{eq:theta0-cos} and \eqref{eq:phi0-cos}, if $\eta(t_\star)=0$ then
  \begin{equation}
  \label{eq:sinus}
  s_{\theta(t_\star)}s_{\phi(t_\star)}=-\frac{4\alpha d(1-\alpha)(1+d)}{1-\alpha+d}<0.
  \end{equation}
  Note that, if $\eta(t_\star)=0$, then \eqref{eq:theta0-cos}, \eqref{eq:phi0-cos} and~\eqref{eq:sinus} are actually equivalent to the condition $\tilde M_{1,1}=\tilde M_{2,1}=0$.
  Hence, to conclude the proof it is enough to show that  if this condition is satisfied at $t=0$, then it is satisfied for $t\in [0,T]$.
  This is an immediate consequence of the fact that, differentiating $\tilde M_{1,1},\tilde M_{2,1}$ along the system~\eqref{eq:dynamics-2d}, one obtains the following linear system
  \begin{equation}
  \dot{\tilde M}_{1,1} = \frac{\eta}{1-\alpha+d} \tilde M_{2,1},\quad \dot{\tilde M}_{2,1} = -\frac{\eta}{1-\alpha+d} \tilde M_{1,1}.
  \end{equation}
\end{proof}

Observe that, for any $(\theta,\eta,\phi)$ satisfying the conditions of Lemma~\ref{lem:dynamics-2d}, the triple $(2\pi-\theta,2\pi-\eta,\phi)$, corresponding to a reflection of $\omega$ and $q$ with respect to the vertical axis, also satisfies such conditions 
and has the same cost.
Note that $\alpha<1$ otherwise $\phi_0=0$ and the corresponding trajectory of \eqref{eq:dynamics-2d}
is constant, contradicting the \eqref{PE0} condition. Hence, without loss of generality, we assume in the sequel that $\phi_0=\phi(0)\in(0,\pi)$.
We next show that the dynamic of the control $c$ is actually independent of $\theta$ and $\eta$.

\begin{proposition}\label{prop:a-b-nu} 
  The control $c=e^{i\phi/2}$, $\phi\in\mathbb{R}$, satisfies the pendulum equation
  \begin{equation}\label{eq:inv-pendulum}
        \ddot\phi = \frac{1}{2\nu^2}s_\phi, \qquad 
        \text{where}\qquad
        \nu = \sqrt{\frac{1-\alpha+d}{2(\alpha+d)}},
  \end{equation}
  with period $2T/\kappa$ for some $\kappa\in \N^*$.
  Moreover, 
  $\eta(t)=0$ if and only if $t=jT/\kappa$, for $j\in \llbracket 0,\kappa\rrbracket $, and the following relations hold
  \begin{equation}\label{eq:a-b}
    a = \nu\kappa K_+(\phi_0),
    \qquad
    b = \nu\kappa K_-(\phi_0),
    \quad\text{where}\quad
    K_\pm(\gamma) = \int_{\gamma}^\pi \frac{1\pm c_\phi}{\sqrt{c_\gamma-c_\phi}}\,d\phi.
  \end{equation}
\end{proposition}

\begin{proof}
From~\eqref{M21} and the second equation in~\eqref{eq:dynamics-2d}  one obtains that $\dot\eta=\frac12 (\alpha+d) s_\phi$. By taking the time derivative of the last equation of \eqref{eq:dynamics-2d}, we then get  \eqref{eq:inv-pendulum}.

Since $\dot Q = cc^\top$, we have $\diag(a,b)=Q(T) = \int_{0}^T cc^\top\,dt$. By simple computations, we have that	  
\begin{equation}	cc^\top = \frac12 \left(\idty_2 + 
  \left(\begin{array}{cc}	c_\phi & s_\phi\\	s_\phi & -c_\phi	\end{array}	\right)	\right).	  \end{equation}	  This yields at once that
	  \begin{equation}\label{eq:boh}
		a = \frac12\left(T +\int_0^T c_\phi\,dt\right)
		\qquad \text{and}\qquad
		b = \frac12\left(T -\int_0^T c_\phi\,dt\right). 
	  \end{equation}
The statement follows by standard facts on the pendulum equation, see \cite{Arnold1978}.
\end{proof}

It is convenient to rewrite the functions $K_\pm$ in terms of classical elliptic integrals. By a simple change of coordinates one obtains 
\begin{align}
K_+(\phi_0) &= 2\sqrt{2} ( K(c_{\phi_0/2})-E(c_{\phi_0/2}), \label{eq:K+E}\\K_-(\phi_0)&= 2\sqrt{2} E(c_{\phi_0/2}),\label{eq:K-E}
\end{align}
where 
\begin{equation}
K(x):=\int_0^{\pi/2}\frac{1}{\sqrt{1-x^2 s_u^2}}\,du,\qquad E(x):=\int_0^{\pi/2}\sqrt{1-x^2 s_u^2}\,du
\end{equation}
are the complete elliptic integrals of the first and  second kind, respectively.
We recall that $K,E$ are monotone functions such that $K(x)\geq E(x)$ for any $x\in [0,1)$ and with equality only if $x=0$. Moreover one has that
\begin{align}
&K(0)=E(0)=\frac{\pi}2,\qquad  \lim_{x\to 1}K(x)=\infty,\qquad E(1)=1,\label{ell-boundary}\\
&\frac{dE}{dx}(x)=\frac{E(x)-K(x)}{x},\qquad\frac{dK}{dx}(x)=\frac{E(x)}{x(1-x^2)}-\frac{K(x)}{x},\label{ell-deriv}\\
&\lim_{x\to 0^+}\frac{K(x)-E(x)}{x^2} = \frac{\pi}4.\label{ell-limit}
\end{align}
We show below that the conditions obtained in Lemma~\ref{lem:dynamics-2d} completely characterize (up to orthogonal transformations of the coordinates) the extremals of~\eqref{ocp0}.
\begin{proposition}\label{prop:char-ext}
Let $(\alpha,d)\in (0,1)\times (0,\infty)$ and assume that $K_+(\phi_0) < K_-(\phi_0)$, where $\phi_0\in (0,\pi)$ is defined by~\eqref{eq:phi0-cos}. 
Then the solutions of~\eqref{eq:dynamics-2d} satisfying the conditions of Lemma~\ref{lem:dynamics-2d} correspond to extremal trajectories $\lambda = (\theta, \diag(a,b),\eta, \alpha,d)$ for~\eqref{ocp0}, for $a=\nu\kappa K_+(\phi_0),b=\nu\kappa K_-(\phi_0),$ for every positive integer $\kappa$. 

On the other hand, for any $0<a<b$ and positive integer $\kappa$, there exists (up to time-invariant orthogonal transformations of the components $(\omega,Q,p)$) a unique extremal trajectory for~\eqref{ocp0} with trace identically equal to one such that $p$ exactly vanishes $\kappa+1$ times on its interval of definition $[0,a+b]$. This trajectory corresponds to a solution of~\eqref{eq:dynamics-2d} for some $(\alpha,d)\in (0,1)\times (0,\infty)$.
\end{proposition}
\begin{proof}
The first part of the proposition easily follows from the results proved above. Indeed, trajectories of~\eqref{eq:dynamics-2d} satisfy $(i)$ in Definition~\ref{def:ext} by definition, while $(iii)$ follows from Lemma~\ref{lem:dynamics-2d} and~\eqref{eq:M2}.  Proposition~\ref{prop:a-b-nu} shows that  $Q(a+b)=\mathrm{diag}(a,b)\in \Sy_n(a,b)$ with $a=\nu\kappa K_+(\phi_0)<b=\nu\kappa K_-(\phi_0),$ for some positive integer $\kappa$. In particular $P_Q$ belongs to the normal cone of $\Sy_n(a,b)$, proving $(ii)$ in Definition~\ref{def:ext}. 

To prove the second part of the proposition, we will establish a one-to-one correspondence between the pairs of positive numbers $(a,b)$ and the possible pairs of parameters $(\alpha,d)$. This is enough to conclude the proof since Lemma~\ref{lem:dynamics-2d} identifies a unique extremal trajectory up to a reflection, and in view of Proposition~\ref{prop:diag} and Proposition~\ref{prop:a-b-nu}.
We first notice that the map $(\nu,\phi_0)\mapsto (\nu\kappa K_+(\phi_0),\nu\kappa K_-(\phi_0))$ is a bijection from $(0,\infty)\times (0,\pi)$ to $(0,\infty)\times (0,\infty)$. Indeed, by~\eqref{ell-boundary}-\eqref{ell-deriv}, one deduces  that the map $\phi_0\mapsto K_+(\phi_0)/K_-(\phi_0)$ is strictly decreasing and, moreover, 
\begin{equation}
\lim_{\phi_0\to 0}\frac{K_+(\phi_0)}{K_-(\phi_0)}=\infty,\quad
\lim_{\phi_0\to \pi}\frac{K_+(\phi_0)}{K_-(\phi_0)}=0.
\end{equation}
We now show that the map $(\alpha,d)\mapsto (\nu,\phi_0)$, with $\phi_0$ and $\nu$ satisfying~\eqref{eq:phi0-cos} and \eqref{eq:inv-pendulum}, is also a bijection from $(0,1)\times(0,\infty)$ to $(0,\infty)\times (0,\pi)$. For this purpose we notice that for any $\phi_0\in (0,\pi),\alpha\in (0,1)$, Equation~\eqref{eq:phi0-cos} is an algebraic equation of degree two in the variable $d$. In particular it admits a unique positive zero $d_{\phi_0}(\alpha) = \frac12(-1+\sqrt{1+4\cot (\phi_0/2)^2\alpha(1-\alpha)})$. Substituting this expression into~\eqref{eq:inv-pendulum} it is easy to see that $\nu$ is a strictly decreasing function of $\alpha$ for any given $\phi_0\in (0,\pi)$, with $\lim_{\alpha\to 0}\nu=\infty$ and $\lim_{\alpha\to 1}\nu=0$.
It follows that for any $(\nu,\phi_0)\in (0,\infty)\times (0,\pi)$ there exists a unique pair $(\alpha,d)\in (0,1)\times(0,\infty)$ such that~\eqref{eq:phi0-cos} and \eqref{eq:inv-pendulum} are satisfied, which concludes the proof of the proposition. 
\end{proof}

\subsection{Proof of Proposition~\ref{prop:2d}}
Taking into account Proposition~\ref{prop:upper-a}, it is enough to establish Proposition~\ref{prop:2d} for sequences $(a_l,b_l)_{l\in\N}$ such that $b_l$ tends to infinity as $l$ tends to infinity. 
Moreover, since we need to upper bound $\mu(a,b,2)$, it is enough to find a control $cc^\top\in \Sy_2^+(a,b,a+b)$ and an initial condition $\xi\in\mathbb S^1$,  
whose cost $J(cc^\top,\xi)$ is indeed smaller than $C_0a/(1+b^2)$ for some universal constant $C_0$. We claim that such a control is provided by %Proposition~\ref{prop:exist}. 
Proposition~\ref{prop:char-ext} in the case $\kappa=1$. Showing this claim simply amounts to compute the cost of such a control and to verify the desired inequality.
In order to do so, we introduce some preliminary estimates.

\begin{lemma}\label{prelim-estim}
Consider the extremal trajectories described in Proposition~\ref{prop:char-ext} with $\kappa=1$. Then, the following hold true,
   \begin{equation}\label{eq:Aal-d}
   \alpha\sim_{b\to\infty}\frac{K_-(\phi_0)^2}{2b^2}s^2_{\phi_0/2},
   \quad\text{and}\quad 
   d\sim_{b\to\infty}\frac{K_-(\phi_0)^2}{2b^2}c^2_{\phi_0/2}.
   \end{equation}
Moreover, there exist two positive constants $C_0,C_1$ independent of $\phi_0$ such that, 
\begin{equation}\label{eq:Aphi1}
   C_0\frac{a}{b}\le c^2_{\phi_0/2}\le C_1\frac{a}{b},
    \end{equation}
and there also exist $\bar{C}_0,\bar{C}_1>0$ and $b_1>0$ such that
\begin{equation}\label{eq:asymp}
  \bar{C}_0 \frac{a}{b^2}\le \frac d{\sqrt\alpha} \le \bar{C}_1 \frac{a}{b^2}, \qquad \forall\  0<a\le b \quad\text{s.t.}\quad b>b_1.
\end{equation}
\end{lemma}
\proof 
To prove \eqref{eq:Aal-d} we notice that the function 
\begin{equation}
F:(\alpha,d)\mapsto \Big(\frac{1}{2\nu^2},\frac{1+c_{\phi_0}}{4\nu^2}\Big)=\Big(\frac{K_-(\phi_0)^2}{2b^2},\frac{(1+c_{\phi_0})K_-(\phi_0)^2}{4b^2}\Big),
\end{equation}
where $c_{\phi_0}$ and $\nu$ are given by~\eqref{eq:phi0-cos} and~\eqref{eq:inv-pendulum}, maps the origin to itself and, as its differential at the origin is given by 
\begin{equation}
DF(0)=\begin{pmatrix}1&1\\0&1\end{pmatrix},
\end{equation}
one can apply the inverse function theorem in a neighborhood $\cal{O}$ of the origin.
Since $b$ goes to infinity if and only if $\nu$ goes to infinity, and this implies that $\alpha,d$ go to zero as well as the value of $F$, we can write 
\begin{equation}
(\alpha,d)\sim_{b\to\infty} DF(0)^{-1} \Big(\frac{K_-(\phi_0)^2}{2b^2},\frac{(1+c_{\phi_0})K_-(\phi_0)^2}{4b^2}\Big)
\end{equation}
which proves~\eqref{eq:Aal-d}. 

To get \eqref{eq:Aphi1}, first notice that $\phi_0$ verifies the constraint 
$K_+(\phi_0)\leq K_-({\phi_0})$ implying by \eqref{eq:K+E} and \eqref{eq:K-E} and the 
properties of $E$ and $K$ that $x=c_{\phi_0/2}$ must belong to an interval $[0,x_*]$ with 
$x_*<1$. On the other hand, by again using \eqref{eq:K+E} and \eqref{eq:K-E}, it 
follows that proving \eqref{eq:Aphi1} is equivalent to show positive lower and upper 
bounds for the function $x\mapsto \frac{x^2E(x)}{K(x)-E(x)}$ with $x\in [0,x_*]$. 
It is then enough to prove that the previous function admits positive limits as $x$ 
tends to zero and $x_*$, which clearly holds true by \eqref{ell-limit}
and the fact that $x_*<1$.  

By taking $b$ large enough and using \eqref{eq:Aal-d}, one gets that there exists 
$b_1>0$ such that, for every $b>b_1$ it holds 
\begin{equation}
\frac12\frac {K_-(\phi_0)c^2_{\phi_0/2}}{\sqrt{2}bs_{\phi_0/2}}\le \frac d{\sqrt\alpha} \le \frac32\frac {K_-(\phi_0)c^2_{\phi_0/2}}{\sqrt{2}bs_{\phi_0/2}}.
\end{equation}
By a reasoning similar to the one yielding \eqref{eq:Aphi1} and using the fact that 
$\frac {K_-(\phi_0)}{s_{\phi_0/2}}$ is uniformly bounded by two positive constants in the admissible range, one deduces \eqref{eq:asymp}.
\endproof

\begin{lemma}\label{le:cost}
 Consider the extremal trajectories described in Proposition~\ref{prop:char-ext} with $\kappa=1$.
Then, there exist $b_0>0$ and two positive constants $C_0,C_1$ such that 
\begin{equation}\label{eq:cost11}
C_0\frac{a}{b^2}\le J(cc^\top,\omega_0)\le C_1\frac{a}{b^2}, \qquad \forall b>b_0,\,  a\le b.
\end{equation}
\end{lemma}

\begin{proof}
From \eqref{ocp0}, one has at once that
\begin{equation}\label{cost-theta-phi}
  J(cc^\top,\omega_0)=\int_0^{a+b}c^2_{(\theta-\phi)/2}dt,
\end{equation}
and $c^2_{(\theta-\phi)/2}=s^2_{\vp/2}$, where $\varepsilon=\theta-\phi-\pi$. From \eqref{eq:dynamics-2d}, the dynamics of $\varepsilon$ on $[0,a+b]$ is given by
\begin{equation}\label{eq:vp}
\ddot\vp=-{\mu}s_\vp+s_\vp c_\vp, 
\quad\text{where}\quad {\mu}=\frac1{1-\alpha+d}.
\end{equation}
Moreover, the initial conditions $(\vp_0,\dot\vp_0)=(\vp(0),\dot\vp(0))$ satisfy the relations %are solutions of
\begin{equation}\label{eq:vp-in}
c_{\vp_0}=1-\frac{2\alpha d}{1-\alpha+d},
\qquad 
\dot\vp_0=-s_{\vp_0}=-(\alpha+d)s_{\phi_0},
\end{equation}
and $(\vp(a+b),\dot\vp(a+b)) = (-\vp_0,-\dot\vp_0)$. 
Notice also that if there exists $t_1\in [0,a+b]$ such that $\vp(t_1)=0$, then $\vp(t_1+t)=-\vp(t_1-t)$ for times $t_1-t,t_1+t$ in $[0,a+b]$.

We have the following first integral for $\vp$ after integrating between the times zero and $t\in [0,a+b$] and taking into account \eqref{eq:vp-in}:
\begin{equation}\label{eq:FIvp}
\dot\vp^2=2{\mu}(c_\vp-c_{\vp_0})+s^2_{\vp}. 
\end{equation}
Since $\vp$ starts at time $t=0$ with negative speed $\dot\vp_0$ 
according to \eqref{eq:vp-in}, $\vp$ will decrease in a right neighborhood of $t=0$. Note that, from \eqref{eq:FIvp}, $\dot\vp$ will keep the same sign, i.e., negative, as long as $\vert \vp\vert\le \vp_0$. Hence, $\vp$ will reach the value $\vp=-\vp_0$ at a time $t_0$, however with a negative speed. Therefore, by \eqref{eq:vp-in}, $t_0$ must be strictly smaller than $a+b$ and $\vp$ decreases in a right neighborhood of $t=t_0$. This will go on till either $\dot \vp=0$ or $\vp=-\pi/2$, since at time $t=a+b$ we have $\varepsilon(a+b)=-\varepsilon_0$ and $\dot\varepsilon(a+b)=-\dot\varepsilon_0$. The latter possibility is clearly ruled out since the r.h.s.\ of \eqref{eq:FIvp} is negative at $\varepsilon = -\pi/2$ for $b$ sufficiently large.
Then, $\dot \vp=0$ occurs at some time $\bar{t}<a+b$ for $\vp=-\bar{\vp}$, where $\bar{\varepsilon}$ is the unique angle in $(0,\pi/2)$ verifying
\begin{equation}\label{eq:barvp}
2{\mu}(c_{\bar{\vp}}-c_{\vp_0})+s^2_{\bar{\vp}}=0.
\end{equation} 
Since $\vp(\bar{t})$ is a minimum for $\vp$, one must necessarily have that 
$\ddot\vp(\bar{t})\geq 0$.
On the other hand, \eqref{eq:vp} can be written $\ddot\vp =-s_\vp({\mu}-c_\vp)$, 
yielding that ${\mu}-c_{\vp(\bar{t})}\geq 0$ since $s_{\vp(\bar{t})}<0$.
We can rewrite \eqref{eq:barvp} as
\begin{equation}\label{eq:ouf0}
({\mu}-c_{\bar{\vp}})^2=1+{\mu^2}-2{\mu}c_{\vp_0}=\frac{(\alpha+d)^2}{(1-\alpha+d)^2},
\end{equation}
which is strictly positive. Then $\bar{t}$ is an isolated zero of $\dot\vp$ and the latter 
must change sign there, implying that $\vp$ increases in a right neighborhood of 
$t=\bar{t}$. By a similar reasoning as before, $\vp$ increases till $\vp=-\vp_0$
at a time ${\tau}$ and one will also get that $\dot\vp({\tau})=-\dot\vp_0=-s_{\vp({\tau})}$. 

We next show that ${\tau}=a+b$. Notice first that $\vp$ is periodic of
period equal to $2\tau$ and moreover there must exists an integer $m$ such 
that $a+b=(2m+1)\tau$. Since $\vp$ satisfies the equation
\begin{equation}
\dot{\vp}=-\vp-\frac{2\eta}{1-\alpha+d}
\end{equation}
one deduces that $\eta$ is periodic with period less than or equal to the one of $\vp$. Finally, recall that the minimal period of $\eta$ coincides with that of $\phi$, which is equal to $2(a+b)$ since $\nu=1$. Hence $2(2m+1)\tau= 2(a+b)\leq 2\tau$ implying that $\tau$ is equal to $a+b$.

To provide an estimate to~\eqref{cost-theta-phi}
 let us first derive an asymptotics for $\bar{\vp}$ as $b$ tends to 
infinity. From \eqref{eq:ouf0}, \eqref{eq:vp}
and the fact that ${\mu}\geq c_{\vp(\bar{t})}$, we have
\begin{equation}
c_{\bar{\vp}}={\mu}-\sqrt{1+{\mu^2}-2{\mu}c_{\vp_0}} = 1-\frac{2d}{1-\alpha+d},
\end{equation}
which yields
\begin{equation}\label{eq:Abarvp}
\bar{\vp}\sim2\sqrt{d}, \qquad\text{as }b\to +\infty.
\end{equation}

By the previous claim, $\vert\vp\vert\le 3\sqrt{d}$ for $b$ large enough on $[0,a+b]$. Subtracting \eqref{eq:barvp} to \eqref{eq:FIvp} yields
\begin{equation}
\dot\vp^2=2{\mu}(c_\vp-c_{\bar{\vp}})-c^2_\vp+c^2_{\bar{\vp}}=(c_\vp-c_{\bar{\vp}})\big(({\mu}-c_\vp)+({\mu}-c_{\bar{\vp}})\big).
\end{equation}
We have ${\mu}-c_{\bar{\vp}}=(\alpha+d)\big(1+O(\alpha)\big)$ and 
${\mu}-c_\vp=(\alpha-d)\big(1+O(\alpha)\big)+2s^2_{\vp/2}$. 
Hence 
\begin{equation}
({\mu}-c_{\bar{\vp}})+({\mu}-c_{\vp})=(2\alpha+\vp^2/2)\big(1+O(\alpha)\big).
\end{equation}
On the other hand, 
\begin{equation}
c_\vp-c_{\bar{\vp}}=\frac{\bar{\vp}^2-\vp^2}2\big(1+O(\alpha)\big).
\end{equation}
Gathering the previous inequalities then yields
\begin{equation}\label{eq:Avp}
\dot\vp^2=(\bar{\vp}^2-\vp^2)(\alpha+(\vp/2)^2)(1+O(\alpha)),
\end{equation}
where $O(\alpha)$ denotes a function of the time such that $|O(\alpha)|\le c\alpha$ on the interval $[0,a+b]$, for some $c>0$ independent of $b$.

We can finally prove the desired estimate for the cost. Indeed, by taking into account the previous results, one has
\begin{equation}
J(cc^\top,\omega_0)=2\int_0^{\bar{\vp}}\frac{s^2_{\vp/2}\,d\vp}{\big(2{\mu}(c_\vp-c_{\bar{\vp}})-c^2_\vp+c^2_{\bar{\vp}}\big)^{1/2}}.
\end{equation}
Using \eqref{eq:Avp}, the above equation can be rewritten as
\begin{equation}
J(cc^\top,\omega_0)\sim\frac{(1+O(\alpha))}2\int_0^{\bar{\vp}}\frac{\vp^2d\vp}{\big((\bar{\vp}^2-\vp^2)(\alpha+\frac{\vp^2}2)\big)^{1/2}},\qquad\text{as } b\to +\infty.
\end{equation}
Thanks to \eqref{eq:Abarvp}, this further simplifies to $J(cc^\top,\omega_0)\sim(1+O(\alpha))\frac{d}{\sqrt{\alpha}}\mathcal J(d/\alpha)$, where
\begin{equation}
\mathcal J(\gamma)=2\int_0^{1}\frac{v^2dv}{\big((1-v^2)(1+2\gamma v^2)\big)^{1/2}}, \qquad \gamma>0.
\end{equation}
Since %$d/\alpha\in (0,1)$,
$d/\alpha$ is bounded (this comes from~\eqref{eq:Aal-d} and the fact that $c_{\phi_0/2}$ belongs to an interval $[0,x_*]$ with $x_*<1$), $\mathcal J(d/\alpha)$ is bounded below and above by positive constants independent of $(a,b)$. Together with the inequalities~\eqref{eq:asymp}, this  concludes the proof of the statement.
\end{proof}

\appendix

\section{A weak-$\ast$ compactness result}\label{a:weakly-star}

For $T>0$, let $L_\infty$ denote the space of essentially bounded real symmetric matrix-valued functions $M:[0,T]\to \R^{n\times n}$. As any two norms are equivalent in finite dimensional spaces, the definition of $L_{\infty}$ is independent of the choice of the norm in the vector subspace of $\R^{n\times n}$ made of real symmetric matrices. The space $L_{\infty}$ can be identified with the dual of the space of integrable real symmetric matrix-valued functions $L_1$, via the duality
\begin{equation}
\langle M, A\rangle = \int_0^T \Tr(M(t) A(t))\,dt, \qquad M\in L_\infty, \, A\in L_1.
\end{equation}
We recall that a sequence $(M_\ell)_{\ell\ge 0}\subset L_\infty$ is weakly-$\ast$ convergent to $M\in L_\infty$ if 
\begin{equation}
\lim_{\ell\to +\infty}\langle M_\ell, A\rangle  = \langle M,A\rangle, \qquad \forall A \in L_1.
\end{equation}

\begin{lemma}\label{lem:weakly-star}
The set
\begin{equation}
\mathfrak S = \{ S\in \Sy_n^+(a,b,T)\mid \Tr S \text{ is constant for a.e.\ } t\in [0,T]\},
\end{equation}
is weakly-$\ast$ compact in $L_\infty$.
\end{lemma}

\begin{proof}
Notice that for any $S\in\mathfrak S$, we have for a.e.\ $\tau\in [0,T]$ that
  \begin{equation}
\|S(\tau)\|_1 = \Tr(S(\tau)) = \frac1T \int_0^T \Tr S(t)\,dt\le \frac{bn}T,
  \end{equation}
where  $\|A\|_1:=  \Tr \sqrt{A^TA}$ is the Schatten $1$-norm of a matrix $A\in\R^{n\times n}$.  
Indeed the first equality comes from the fact that 
$S$ takes positive semi-definite values, the second one  from the definition of $S\in\mathfrak S$, and the last inequality from Condition \eqref{PE00}. This shows that $\mathfrak S$ is a bounded subset of $L_\infty$. By 
Banach-Alaoglu theorem, it then suffices to prove that $\mathfrak S$ is weak-$\ast$ closed in $L_\infty$.
%Let us show that $\mathfrak S$ is weakly-$\ast$ closed in $L_\infty$. 
To this effect, let $(S_\ell)_{\ell\ge 0}$ a sequence in $\mathfrak S$ which  weakly-$\ast$ converges to $S\in L_\infty$. For $\ell\geq 0$, let $\mathcal{T}_\ell$ be the constant value taken by the function $t\mapsto \Tr S_\ell(t)$. Since  $\mathcal{T}_\ell=\langle S_\ell, \idty_n\rangle/T$ for $l\geq 0$, it follows that  $\lim_{\ell\to\infty}\mathcal{T}_\ell= \langle S, \idty_n\rangle/T$.

For every $\varepsilon\in (0,T)$ and $\tau\in [\varepsilon,T-\varepsilon]$, 
let $\chi_\varepsilon^\tau:[0,T]\to \{0,1\}$ be the characteristic function of the interval $[\tau-\varepsilon,\tau+\varepsilon]$. Then, by Lebesgue theorem and the definition of weak-$\ast$ convergence, it holds for a.e.\ $\tau\in [0,T]$ 
\begin{equation}
\Tr S(\tau)= \lim_{\varepsilon\downarrow0} \frac1{2\varepsilon} \langle S, \chi_\varepsilon^\tau \idty_n\rangle = \lim_{\varepsilon\downarrow0} \frac1{2\varepsilon} \lim_{\ell\to +\infty} \langle S_\ell, \chi_\varepsilon^\tau \idty_n\rangle.
\end{equation}
Since $S_\ell\in \mathfrak S$ we have that $\langle S_\ell, \chi_\varepsilon^\tau \idty_n\rangle = 2\varepsilon \mathcal{T}_\ell$ if $\tau\in [\varepsilon,T-\varepsilon]$, which finally yields that, for a.e. in $\tau\in [0,T]$,
$\Tr S(\tau)$ is equal to the constant $\langle S, \idty_n\rangle/T$.

We are left to show that $S\in \Sy_n^+(a,b,T)$. 
We have that 
\begin{equation}
 a\|x\|^2 \le x^\top \left(\int_0^T  S_\ell(t)dt\right)x = \langle S_\ell, xx^\top\rangle \le b\|x\|^2
\end{equation}
for any  $x\in \R^n$ and $\ell\geq 0$, and passing to the limit as $\ell$ goes to infinity we obtain $a\idty_n\le \int_0^T S(t)dt\le b\idty_n$.
Moreover, again by Lebesgue theorem, for a.e.\ $\tau\in[0,T]$ and $x\in \R^n$ we have
\begin{equation}
x^\top S(\tau)x = \Tr(S(\tau)xx^\top) = \lim_{\varepsilon\downarrow 0} \frac1{2\varepsilon}\lim_{\ell\to +\infty}\langle S_\ell, \chi_\varepsilon^\tau xx^\top\rangle.
\end{equation}
Since $\langle S_\ell, \chi_\varepsilon^\tau xx^\top\rangle = \int_0^T \chi_\varepsilon^\tau(t)x^\top S_\ell(t)xdt\ge 0$, we obtain that  $x^\top S(\tau)x\geq 0$ for a.e.\ $\tau\in[0,T]$. This completes the proof of the lemma.
\end{proof}

\section{Hamiltonian equations}\label{a:ham}

In this appendix, we apply the Pontryagin Maximum Principle (PMP) to the control system \eqref{ocp0}, in order to derive necessary optimality conditions. These are essential to the proofs of Proposition~\ref{prop:pmp}.

For the Hamiltonian formalism used below, we refer to \cite{Agrachev-book}.

\begin{lemma}\label{lem:ham-vect}
  Let $H\in C^\infty(T^*\mathbb{S}^{n-1})$ be an Hamiltonian function. Upon the identification $T_\omega^*\mathbb{S}^{n-1}\simeq \omega^\perp$, the corresponding Hamiltonian system $\dot \xi = \vec H(\xi)$, $\xi=(\omega, p)\in T^*\mathbb{S}^{n-1}$, reads
  \begin{eqnarray}
    \dot\omega &=&  \frac{\partial H}{\partial p}-\left(\omega^\top\frac{\partial H}{\partial p}\right)\omega\\
    \dot p &=& \frac{\partial H}{\partial \omega} - \left(\omega^\top\frac{\partial H}{\partial \omega}\right)\omega - \left(\omega^\top\frac{\partial H}{\partial p}\right)p + \left(p^\top\frac{\partial H}{\partial p}\right)\omega.
  \end{eqnarray}
\end{lemma}

\begin{proof}
  Upon the given identifications, we have that
  \begin{equation}\label{eq:tang-cotang}
    T_{(\omega,p)}(T^*\mathbb{S}^{n-1}) = \left\{ (v_1,v_2)\in \R^{2n}\mid  \omega^\top v_1 =0 \quad\text{and}\quad  p^\top v_1 + \omega^\top v_2 = 0 \right\}.
  \end{equation}
  Letting $(\frac{\partial H}{\partial \omega},\frac{\partial H}{\partial p})\in \R^{2n}$ be the partial derivative at $\xi=(\omega,p)\in\mathbb S^{n-1}$ of $H$, we have
  \begin{equation}
    d_{\xi}H(v_1,v_2) = v_1^\top \frac{\partial H}{\partial \omega} + v_2^\top \frac{\partial H}{\partial p}, \qquad \forall(v_1,v_2)\in T_{(\omega,p)}(T^*\mathbb{S}^{n-1}).
  \end{equation}
  On the other hand, the Hamiltonian vector field $\vec H\in \Gamma(T^*\mathbb{S}^{n-1})$, with components $\vec H = (\vec H_p,-\vec H_\omega)\in T(T^*\mathbb{S}^{n-1})$, is the only vector field such that
  \begin{equation}
    d_{\xi} H(v_1,v_2) = v_1^\top \vec H_\omega(\xi)+v_2^\top \vec H_p(\xi), \qquad \forall(v_1,v_2)\in T_{(\omega,p)}(T^*\mathbb{S}^{n-1}).
  \end{equation}
  As a consequence of these two facts, we have
  \begin{gather}
    \label{eq:vvv0}
    \omega^\top \vec H_p(\omega,p) = 0, 
    \qquad
    p^\top \vec H_p(\omega,p) = \omega^\top \vec H_\omega(\omega,p),
    \\
    \label{eq:vvv}
    v_1^\top \left( \frac{\partial H}{\partial \omega} - \vec H_\omega(\omega,p) \right) + v_2^\top\left(\frac{\partial H}{\partial p} - \vec H_p(\omega,p)\right)=0, \qquad \forall(v_1,v_2)\in T_{(\omega,p)}(T^*\mathbb{S}^{n-1}).
  \end{gather}
  By \eqref{eq:tang-cotang}, we have that, if $(v_1,v_2)\in T_{(\omega,p)}(T^*\mathbb{S}^{n-1})$ is such that $\omega^\top v_2 = 0$, then $v_1=0$. As a consequence, considering \eqref{eq:vvv} for such $(v_1,v_2)$ and taking into account the first equation of \eqref{eq:vvv0}, yields
  \begin{equation}
     \vec H_p(\omega,p) = \frac{\partial H}{\partial p}-\left(\omega^\top\frac{\partial H}{\partial p}\right)\omega.
  \end{equation}
  Plugging this in \eqref{eq:vvv}, one deduces that
  \begin{equation}
    \vec H_\omega(\omega,p)= \frac{\partial H}{\partial \omega} - \left(\omega^\top\frac{\partial H}{\partial \omega}\right)\omega - \left(\omega^\top\frac{\partial H}{\partial p}\right)p + \left(p^\top\frac{\partial H}{\partial p}\right)\omega.
  \end{equation}
  This completes the proof.
\end{proof}

\begin{proposition}\label{prop:ham}
  Let $(\omega,Q):[0,T]\to \mathcal M$ be an optimal trajectory of system \eqref{ocp0}, associated with an optimal control $S$. Then, there exists a curve $t\in [0,T]\mapsto (p(t),P_Q(t)) \in T^*_{\omega(t)}\mathbb S^{n-1}\times T^*_{Q(t)}\Sy_n$ and $\nu_0\in\{0,1\}$, satisfying \eqref{eq:adj1}, \eqref{eq:adj2}, \eqref{eq:max}, and \eqref{eq:trans}.
\end{proposition}

\begin{proof}
  Let $\lambda = (\omega, Q, p, P_Q)$. Recall that we consider the identification $T^*_{\omega}\mathbb S^{n-1}\times T^*_{Q}\Sy_n \simeq (\mathbb{R}\omega)^\perp\times \Sy_n$. We also let $T_{(\omega,Q)}\mathcal M \simeq (\mathbb{R}\omega)^\perp\times \Sy_n$, so that
  \begin{equation*}
    \langle \lambda, v\rangle = p^\top v + \Tr(P_Q V) = \Tr(v p^\top + VP_Q), \quad \text{for every }v= (v, V)\in T_{(\omega, Q)}\mathcal M.
  \end{equation*}
  We follow the formulation of the PMP given in \cite[Theorem~12.4]{Agrachev-book}.
  Simple computations show that, for $\nu_0\in\{0,1\}$, the Hamiltonian associated with the system is given by (up to constants)
  \begin{equation}
    H(\lambda, S) = \frac{\Tr(S \tilde M)}2 
    ,\quad\text{where}\quad 
    \tilde M = P_Q - \left( \omega p^\top + p\omega^\top + \nu_0\omega\omega^\top \right).
  \end{equation}

  Equations \eqref{eq:max} and \eqref{eq:trans} are immediate consequences of the PMP. In order to complete the proof, we are hence left to check \eqref{eq:adj1} and \eqref{eq:adj2}. By the PMP, we have $\dot\lambda = \vec{H}$, where we let $\vec H\in\Gamma(T^*\mathcal M)$ be the Hamiltonian vector field associated with the control $S$. 
  
  We observe that the Hamiltonian decomposes as $H(\omega, Q, p, P_Q,S) = H_1(\omega, p, S)+H_2(Q,P_Q,S)$. This implies that a similar decomposition holds for the corresponding Hamiltonian vector field. Thus, \eqref{eq:adj1} follows by Lemma~\ref{lem:ham-vect} and the fact that
  \begin{equation}
    \frac{\partial H_1}{\partial \omega} = - Sp -\nu_0 S\omega,
    \quad\text{and}\quad
    \frac{\partial H_1}{\partial p} =-S\omega.
  \end{equation}
  On the other hand, \eqref{eq:adj2} follows from the easily verified fact that
  \begin{equation}
    \vec H_2 = \left( \frac{\partial H_2}{\partial P_Q}, -\frac{\partial H_2}{\partial Q}\right) = (S,0).
  \end{equation}
  This concludes the proof of the statement.
\end{proof}

\bibliography{biblio}
\bibliographystyle{siamplain}

\end{document}